\documentclass{amsart}

\usepackage{amsmath, amssymb, amsthm, mathtools, tikz-cd}
\usepackage[hidelinks]{hyperref}
\usepackage{environ}
\usepackage{comment}
\usepackage{enumitem}
\usepackage{url}

\theoremstyle{plain}
\newtheorem{thm}{Theorem}[section]
\newtheorem{prop}[thm]{Proposition}

\newtheorem{cor}[thm]{Corollary}
\newtheorem{lem}[thm]{Lemma}

\theoremstyle{definition}
\newtheorem{defi}[thm]{Definition}

\newtheorem*{acknowledgments}{Acknowledgments}

\theoremstyle{remark}
\newtheorem{rmk}[thm]{Remark}

\newcommand{\cH}{\mathcal{H}}
\newcommand{\cF}{\mathcal{F}}

\newcommand{\cL}{\mathcal{L}}
\newcommand{\bC}{\mathbb{C}}

\newcommand{\bN}{\mathbb{N}}
\newcommand{\bZ}{\mathbb{Z}}
\newcommand{\bR}{\mathbb{R}}
\newcommand{\op}{\text{op}}

\newcommand{\id}{\operatorname{id}}
\newcommand{\End}{\operatorname{End}}

\newcommand{\tor}{\operatorname{tor}}
\newcommand{\homeo}{\operatorname{Homeo}}

\newcommand{\clo}[1]{\overline{#1}}
\newcommand{\Span}{\operatorname{span}}
\newcommand{\Tr}{\operatorname{Tr}}

\newcommand{\restrict}[2]{\left.{#1}\right|_{#2}}
\newcommand{\norm}[1]{\left\lVert {#1}\right\rVert}

\newcommand{\AbCAlg}{\mathsf{AbC^*Alg}}
\newcommand{\Prim}[1]{\operatorname{Prim}{#1}}
\DeclareMathOperator{\Hilb}{\mathsf{Hilb}}

\title{The Homotopy 3-Type of Abelian C*-Algebras}
\author{Gregory Faurot}
\address{Department of Mathematics, The Ohio State University, Columbus, OH 43210}
\email{faurot.3@osu.edu}

\author{Giovanni Ferrer}
\email{ferrer.40@osu.edu}

\date{\today}


\begin{document}

\begin{abstract}
    We compute the homotopy groups at each unital abelian C*-algebra $C(T)$ in the Morita $3$-category of abelian C*-algebras, C*-algebras with central maps, C*-correspondences, and adjointable bimodule maps. We describe these groups in terms of the topological data of the underlying compact Hausdorff space $T$. We also compute the actions of the first homotopy group on the second and third homotopy groups in terms of these topological invariants of $T$. 
\end{abstract}
\maketitle
\setcounter{tocdepth}{1}
\tableofcontents

\addtocontents{toc}{\protect\setcounter{tocdepth}{-1}}
\section*{Introduction}
\addtocontents{toc}{\protect\setcounter{tocdepth}{1}}
\renewcommand{\thethm}{\Alph{thm}}
The study of C*-algebras is frequently referred to as ``noncommutative topology" since Gelfand duality describes the correspondence between abelian C*-algebras and compact Hausdorff spaces. This allows topological properties to be reformulated in the context of a (generally noncommutative) C*-algebra. The most notable example of this transfer is the introduction of topological $K$-theory to the study of $C^*$-algebras, which led to the Elliott Classification Program for simple, amenable C*-algebras \cite{CGSTW23}. In a similar vein, the Serre--Swan Theorem (\cite[p. 267]{S62}) gives another correspondence between abelian C*-algebras and topology; in this case, however, it is between finitely generated projective $C(T)$-modules and vector bundles over a compact Hausdorff space $T$.
\par Categorically, Gelfand duality is a statement at the level of $1$-categories, giving an equivalence between the C*-algebraic and topological categories. The Serre--Swan Theorem, on the other hand, can be interpreted as a $2$-categorical statement. Because the C*-$2$-category $\mathsf{C^*Alg}$ has bimodules as its $1$-morphisms, the Serre--Swan Theorem uses topology to describe the finitely generated projective $1$-morphisms from $C(T)$ to itself. In this paper, we explore a $3$-categorical analogue of Gelfand duality and the Serre--Swan Theorem.
\par To elaborate more on C*-categories, the category $\mathsf{C^*Alg}_1$, whose objects are C*-algebras and whose morphisms are $\ast$-homomorphisms, is a $1$-category. C*-algebras also lie inside of a $2$-category $\mathsf{C^*Alg}$, where the objects are C*-algebras, the $1$-morphisms Hilbert C*-bimodules, and the $2$-morphisms adjointable bimodule maps. Furthermore, $\mathsf{C^*Alg}$ is actually a C*-$2$-category, where the sets of $2$-endomorphisms have the structure of C*-algebras. By considering abelian C*-algebras as $E_2$-algebras in the category of vector spaces $\mathsf{Vect}$, abelian C*-algebras form a Morita $3$-category of $E_2$-algebras called $\AbCAlg$. This category turns out to be an example of a C*-$3$-category and was first investigated by the second-named author in \cite{Ferrer24}.
\par We therefore look to the $3$-category $\AbCAlg$ for a $3$-categorical correspondence between topology and abelian C*-algebras. To that end, the homotopy hypothesis of Grothendieck (\cite{G83}) states that there should be an equivalence between homotopy $n$-types and (weak) $n$-groupoids. Thus, one possible $3$-categorical approach would be to compute the homotopy groups at abelian C*-algebras in the $3$-category $\AbCAlg$, thus describing the homotopy $3$-type. Corey Jones suggested that the first homotopy group at a unital abelian C*-algebra in $\AbCAlg$ should decompose as a short exact sequence involving $H^3(T;\bZ)$ and ${\homeo}(T)$. Our first theorem proves this statement; using the Serre--Swan theorem, it also describes the other homotopy groups in terms of topological invariants of the compact Hausdorff space $T$. We emphasize that these are not the traditional homotopy groups of $T$ from algebraic topology.
\begin{thm}\label{thmA}
    Let $C(T)$ be an abelian C*-algebra. Then the homotopy groups at $C(T)$ in $\AbCAlg$ are as follows:
    \begin{align*}
        \pi_0(\AbCAlg) &\cong  \{\text{Compact Hausdorff Spaces}\}/\cong,\\
        \pi_1(\AbCAlg,C(T))&\cong  H^3_{\tor}(T;\bZ)\rtimes \operatorname{Homeo}(T),\\
        \pi_2(\AbCAlg,C(T)) &\cong \operatorname{Pic}(T), \text{ and}\\
        \pi_3(\AbCAlg,C(T))&\cong C(T)^\times.
    \end{align*}
    Furthermore, if $T$ has the homotopy type of a CW-complex, 
    $$\pi_2(\AbCAlg,C(T))\cong H^2(T;\bZ).$$
\end{thm}
\noindent Here, $H^3_{\tor}(T;\bZ)$ is the torsion subgroup of the third \v{C}ech cohomology group of $T$ and $\operatorname{Pic}(T)$ is the Picard group of isomorphism classes of complex line bundles over $T$. 
\par However, a homotopy $3$-type is classified by more data than just the homotopy groups. We compute part of this additional information: the actions of $\pi_1$ on the higher homotopy groups $\pi_2$ and $\pi_3$.
\begin{thm}\label{thmB}
    Let $C(T)$ be an abelian C*-algebra with the homotopy groups as described in Theorem~\ref{thmA}. Then the actions of $\pi_1(\AbCAlg,C(T))$ on the homotopy groups $\pi_2(\AbCAlg, C(T))$ and $\pi_3(\AbCAlg, C(T))$ are described as follows: Given a $3$-cocycle $\omega \in H^3_{\tor}(T;\bZ)$, a homeomorphism $\Phi \in {\homeo}(T)$, and a line bundle $E$, we have
    $$(\omega, \Phi)\curvearrowright E \cong (\Phi^{-1})^*(E),$$
    where $(\Phi^{-1})^*(E)$ is isomorphic to the pullback bundle along $\Phi^{-1}$. When $T$ has the homotopy type of a CW-complex (and so $\operatorname{Pic}(T)\cong H^2(T;\bZ)$), this action corresponds to the pullback on $H^2(T;\bZ)$ by $\Phi^{-1}$. If we furthermore have a $3$-morphism $f \in C(T)^\times$, we have
    $$(\omega, \Phi) \curvearrowright f = f \circ \Phi^{-1}.$$
    
\end{thm}
\par Future applications of this work involve the interplay between quantum symmetries and C*-algebras. Classical symmetries of C*-algebras arise as group actions. The study of group actions on C*-algebras has been quite fruitful, such as through the construction of crossed-product C*-algebras (\cite{W07}) or the classification of group actions on simple purely infinite C*-algebras (\cite{GS24}). From a categorical viewpoint, this is because the $\ast$-automorphisms of a C*-algebra form a $1$-category. As C*-algebras naturally lie inside the $2$-category $\mathsf{C^*Alg}$, quantum symmetries of C*-algebras may be obtained from actions of tensor categories on C*-algebras, which is a promising active area of research (e.g. \cite{K24, AKK24, EGPJ24}). By considering abelian C*-algebras inside a $3$-category, we may construct actions by monoidal $2$-categories to obtain higher quantum symmetries.
\par The paper is laid out as follows. In Section~\ref{section:background}, we recount a wide variety of background, both on C*-algebras and C*-categories. In Section~\ref{section:homotopygroups}, we compute the homotopy groups at $C(T)$, proving Theorem~\ref{thmA}. Theorem~\ref{thmB} is proven in Section~\ref{section:actions}. Finally, in order to compute $\pi_2$ in $\AbCAlg$, we needed to know that the Serre--Swan Theorem is a monoidal equivalence. As we were unable to locate a reference for this fact, we provide a proof in Appendix~\ref{Sec:Serre-Swan}.

\begin{acknowledgments}
    The authors would like to thank Corey Jones for initially suggesting this problem. We also thank David Penneys and Nick Gurski for numerous helpful conversations about bimodules and homotopy types. The authors were partially supported by NSF grant DMS-2154389.
\end{acknowledgments}

\setcounter{thm}{0}
	\numberwithin{thm}{section}
\section{Background}\label{section:background}
Unless otherwise stated, we will assume all C*-algebras to be unital.
\begin{defi}
A C*-algebra $A$ consists of the following data.
\begin{itemize}
\item A (unital) algebra $A$. We will use lowercase $a,b,\hdots$ for elements of $A$ and $1 \in A$ for its unit.
\item A (conjugate-linear) involution $\dag \colon \overline{A} \to A$ satisfying $(ab)^\dag = b^\dag a^\dag$ and $a^{\dag \dag} = a$ (note that we are using $\dagger$ for the involution instead of the usual $\ast$ to avoid confusion with precomposition).
\item A complete submultiplicative norm $\|\cdot\|$ on $A$.
\end{itemize}
We require that this data satisfy the C*-identity:
$$
\|a\| = \|a^\dag a\|^{1/2} \qquad \forall a \in A.
$$
\end{defi}
We also recall the definition of positive elements in a C*-algebra. An element $a \in A$ is \emph{positive} ($a\geq0$) if there is a $b \in A$ with $b^\dagger b = a$. We use $A^+$ to denote the set of positive elements.

\subsection{Imprimitivity Bimodules} 
We will use $\langle \cdot | \cdot \rangle$ to denote ($A$-valued) inner products that are linear in the second coordinate. On the other hand, we will use $\langle \cdot, \cdot\rangle$ for inner products that are linear in the first coordinate.
\par A right \emph{inner product C*-module} $X$ is a vector space over $\bC$ that additionally carries a right action of a C*-algebra $A$ that is compatible with the vector space structure. Furthermore, there is an $A$-valued inner product $\langle \cdot | \cdot \rangle\colon X \times X \to A$ satisfying the following properties:

\begin{center}
    \begin{tabular}{ r@{$\null = \null$} l c }
        $\langle x| \alpha y + \beta z\rangle \phantom{l}$& $\alpha \langle x| y \rangle + \beta \langle x | z \rangle$ &$(x,y,z \in X, \alpha,\beta \in \bC)$,\\
        $\langle x | y\triangleleft a \rangle \phantom{l}$& $\langle x | y \rangle a$ &$(x,y \in X, a \in A)$,\\
        $\langle y| x \rangle \phantom{l}$& $\langle x | y \rangle ^\dagger$ &$(x,y \in X)$, \\
        \multicolumn{3}{c}{$\langle x | x \rangle \phantom{l}\geq 0, \text{$\phantom{llllll}$and if } \langle x | x \rangle = 0, \text{ then } x =0$}
        
    \end{tabular}
\end{center}
Any inner product C*-module can be equipped with the following norm:
$$\Vert x \Vert_X = \norm{\langle x | x \rangle}_A^{1/2}.$$
An inner product C*-module is called a (right) \textit{Hilbert C*-module} if it is complete in this norm. A left Hilbert C*-module is defined similarly with a left module action of $A$ on $X$ and an $A$-valued inner product $\langle \cdot, \cdot \rangle$.
\par Observe that $\langle X | X \rangle\coloneq \Span\{\langle x | y \rangle: x,y \in X\}$ is a two-sided ideal of $A$. A Hilbert $A$-module $X$ is called \textit{full} if $\langle X | X \rangle$ is a dense ideal of $A$. The following remark about full Hilbert modules will be used numerous times.
\begin{rmk}\label{fullimpliesfaithful}
We note that fullness implies that the action is faithful. Indeed, for $X_A$ a full Hilbert module, if $x \lhd a = 0$ for every $x \in X$,
$$\langle y | x \rangle_A a = \langle y | x \lhd a \rangle_A = 0.$$
As this holds for all $x,y \in X$, and since $\langle X | X \rangle_A \subseteq A$ is dense, we have $a = 1_Aa  = 0$.
\end{rmk}
\par Given two right Hilbert $A$-modules $X$ and $Y$, a function $f\colon X \to Y$ is \textit{adjointable} if there is another function $f^\dag\colon Y \to X$ satisfying
$$\langle fx| y \rangle = \langle x | f^\dag y \rangle \text{ for all $x \in X$, $y \in Y$}.$$
Adjointable functions are automatically continuous $A$-module maps. We denote the set of adjointable $A$-module maps from $X$ to $Y$ by $\cL_A(X,Y)$, or simply $\cL_A(X)$ if $X=Y$. 
\par The $3$-category $\AbCAlg$ will have bimodules of C*-algebras as its $2$-morphisms, usually called C*-correspondences. An \emph{$A$-$B$ correspondence} is a right Hilbert $B$-module $X$ along with a $\ast$-homomorphism $\phi\colon A \to \cL_B(X)$. The following definition describes the invertible correspondences.
\begin{defi}
    An $A$-$B$ correspondence $X$ is called an \textit{imprimitivity bimodule} if:
\begin{enumerate}
    \item $X$ is a full left Hilbert $A$-module and a full right Hilbert $B$-module.
    \item For all $x,y \in X, a \in A$, and $b \in B$, we have
    $$\langle a\triangleright x | y\rangle_B=\langle x | a^\dagger\triangleright  y \rangle_B \phantom{ll}\text{ and }\phantom{ll} {_A}\langle x\triangleleft b , y\rangle ={_A} \langle x , y\triangleleft b^\dagger \rangle$$
    so that $A$ and $B$ act as adjointable operators relative to the other's inner product.
    \item For all $x,y,z \in X$, we have
    $${_A}\langle x , y \rangle \triangleright z = x \triangleleft \langle y | z\rangle_B.$$
\end{enumerate}
\end{defi}
Imprimitivity bimodules induce an equivalence relation on C*-algebras, analogous to Morita equivalence in algebra.
\begin{defi}
    Two C*-algebras $A$ and $B$ are said to be \emph{Morita equivalent} if there exists an $A$-$B$ imprimitivity bimodule.
\end{defi}

\subsection{Spectra \& Dauns--Hofmann}
We now provide some background on the representation theory of C*-algebras. 
\begin{defi}
    A \emph{representation} of a C*-algebra $A$ is a $\ast$-homomorphism $\lambda\colon A \to B(\cH_\lambda)$ for some Hilbert space $\cH_\lambda$. Two representations $\lambda$ and $\rho$ of $A$ are unitarily equivalent if there is a unitary $(u^\dag=u^{-1})$ operator $u\colon \cH_\rho\to \cH_\lambda$ so that
    $$\lambda(a)=u\rho(a)u^\dagger$$
    for all $a \in A$.
\end{defi}
Of particular importance are the irreducible representations, as these form the building blocks of larger representations. A closed subspace $X\subseteq \cH_\lambda$ is \emph{invariant} for $\lambda$ if $\lambda(a)x \in X$ for all $a \in A$ and $x \in X$.
\begin{defi}
    A representation $\lambda$ is said to be \emph{irreducible} if the only closed invariant subspaces are $\{0\}$ and $\cH_\lambda$. Equivalently, the only operators in $B(\cH_\lambda)$ commuting with $\lambda(A)$ are scalar multiples of the identity by \cite[Lemma A.1]{RW98}.
\end{defi}

\begin{rmk}
    For an irreducible representation $\lambda$ of $Z$, observe that $\lambda(Z(A))\subseteq \bC1_{H_\lambda}$ because $\lambda(Z(A))$ commutes with $\lambda(A)$. As a result, when we write $\restrict{\lambda}{Z(A)}$, we will also implicitly restrict the codomain to be $\bC \cong \bC1_{\cH_\lambda}$.
\end{rmk}

\begin{defi}
    The \emph{spectrum} $\hat{A}$ of a C*-algebra $A$ is defined to be the set of unitary equivalence classes of irreducible representations of $A$.
\end{defi}
Currently, the spectrum is simply a set; it does not carry a topology. A topology on $\hat{A}$ is defined using the ideals that arise as kernels of these irreducible representations.
\begin{defi}
    A closed, two-sided ideal $P$ of $A$ is \emph{primitive} if $P$ is the kernel of an irreducible representation of $A$. Let $\Prim{A}$ denote the set of primitive ideals of $A$.
\end{defi}
The topology on $\Prim{A}$ is defined in the following manner. 
\begin{defi}
    Given a subset $F\subseteq \Prim{A}$, define its closure $\clo{F}$ to be the set
    $$\clo{F}\coloneq \{P \in \Prim{A}:\bigcap_{I \in F}I \subseteq P\}$$
\end{defi}
One verifies, using the Kuratowski closure axioms, that this closure operation defines a topology on $\Prim{A}$ called the \textit{hull-kernel topology}. Since two unitarily equivalent representations have the same kernel, the map $[\lambda]\mapsto \ker\lambda$ is a well-defined map $\ker\colon \hat{A}\to \Prim{A}$. The spectrum is thus equipped with a topology by pulling back the topology on $\Prim{A}$ via the kernel map. In general, neither of these topologies is Hausdorff. However, as we will be interested in unital C*-algebras, the spectrum and primitive ideal space will at least be compact.
\begin{prop}[{\cite[Lemma~A.30]{RW98}}]\label{prop:spectrumcompact}
    Let $A$ be a unital C*-algebra. Then $\hat{A}$ and $\Prim{A}$ are compact.
\end{prop}
\par We now discuss the Dauns--Hofmann Theorem. Since we are only interested in the unital case, we can avoid discussing multiplier algebras.
\begin{thm}[{\cite[Lemma 8.15]{DH68}}, Dauns--Hofmann]\label{DHthm}
    Let $A$ be a unital C*-algebra. Then there is a $\ast$-isomorphism $\mathcal{F}\colon C(\Prim{A}) \to Z(A)$ so that 
    $$q_P(\mathcal{F}(f)a)=f(P)q_P(a)$$
    for all $P \in \Prim{A}$ and $a \in A$, where $q_P\colon A \to A/P$ is the quotient map.
\end{thm}
As a consequence, we obtain a diagram of the following form.
\begin{cor}[c.f. Lemma~\ref{restrict}]
Let $A$ be a unital C*-algebra. Then there exist continuous surjections, as in the following diagram:
\begin{center}
    \begin{tikzcd}
        \hat{A} \arrow[r, two heads,  "\ker"] & \Prim{A} \arrow[r, two heads, "\text{D--H}"] & \widehat{Z(A)}.
    \end{tikzcd}
\end{center}
Moreover, these map witness $\widehat{Z(A)}$ as the Stone--\v{C}ech compactifications\footnote{Usually, this construction is applied to a Hausdorff space $X$ and is commonly referred to as the compactification $\beta X$. In this case, however, $\hat{A}$ is already compact, but not Hausdorff. Thus, this construction would be aptly named the ``Hausdorffication'' $\beta \hat{A}$ of $\hat{A}$. We will further discuss this construction in the later section $\S$\ref{subsec:pi0}.} of $\Prim{A}$ and $\hat{A}$ \cite[Lemma~1.1]{N95}.
\end{cor}
\subsection{Continuous-Trace C*-Algebras}
Before discussing continuous-trace C*-algebras, we need to address the interaction between Morita equivalence and spectra of C*-algebras.
\begin{thm}[{\cite[Theorem~3.29]{RW98}}]\label{moritaspectra}
    Let $A$ and $B$ be Morita equivalent C*-algebras. Then $\hat{A}\cong \hat{B}$.
\end{thm}

\begin{thm}[{\cite[Corollary 3.33]{RW98}}]
    Suppose $X$ is an $A$-$B$ imprimitivity bimodule. Then $X$ induces a homeomorphism $h_X\colon \Prim{B} \to \Prim{A}$ called the Rieffel homeomorphism. The map $h_X$ is given by, for any $P \in \Prim{B}$, 
    $$h_X(P)=\clo{\Span}\{{_A}\langle x\triangleleft b | y\rangle: x, y \in X, b \in P\}.$$
\end{thm}
Since imprimitivity bimodules induce homeomorphisms of the spectra and primitive ideal spaces, we may expect that the Dauns--Hofmann Theorem can be used to understand the actions of $Z(A)$ and $Z(B)$ on ${_A}X_B$. This is made precise in the following proposition.
\begin{prop}[{\cite[Proposition 5.7]{RW98}}]\label{rieffelhomeoactions}
    Suppose $X$ is an $A$-$B$ imprimitivity bimodule, and $A$ and $B$ are unital C*-algebras with Hausdorff spectrum. With $h_X\colon \Prim{B} \to \Prim{A}$ the Rieffel homeomorphism, we have, for all $f \in C(\Prim{A})$ and $x \in X$,
    $$\mathcal{F}(f)\triangleright x = x\triangleleft \mathcal{F}(f \circ h_X).$$
\end{prop}
Because the Rieffel homeomorphism intertwines the central actions, we can also consider equivariant bimodules where this homeomorphism is the identity. 
\begin{defi}\label{bimodoverT}
    Let $A$ and $B$ be C*-algebras with $\ast$-isomorphisms $\phi\colon C(T) \to Z(A)$ and $\psi\colon C(T) \to Z(B)$. Furthermore, assume $\hat{A}$ and $\hat{B}$ are Hausdorff. We say an imprimitivity $A$-$B$ bimodule $X$ is an imprimitivity bimodule \emph{over $T$} if, for all $f \in C(T)$ and $x \in X$,
    $$\phi(f)\triangleright x = x\triangleleft \psi(f).$$
\end{defi}

\begin{prop}[{\cite[Proposition 5.7]{RW98}}]
    With the hypotheses of Definition~\ref{bimodoverT}, $X$ is an $A$-$B$ imprimitivity bimodule over $T$ if and only if the following diagram commutes:
    \begin{center}
    \begin{tikzcd}
        \Prim{B} \arrow[r, "h_X"] \arrow[d, "\hat{\psi}"'] & \Prim{A} \arrow[dl, "\hat\phi"] \\
        T & 
    \end{tikzcd}
    \end{center}
\end{prop}

We are now ready to work towards the definition of continuous-trace C*-algebras. We first need to describe how traces on irreducible representations interact with elements of $A$.

\begin{defi}
    Let $A$ be a C*-algebra and $a \in A^+$. We define the function $\Tr_a\colon \hat{A} \to \bR^+\cup \{\infty\}$ by 
    $$[\lambda] \mapsto \Tr(\lambda(a))$$
    where $\Tr$ is the (unnormalized) trace on $B(\mathcal{H_\lambda})$. In the special case when $a=1_A$, we write $\dim_A$ instead of $\Tr_{1_A}$, as the trace of the unit detects the dimension of the representation.
\end{defi}

In general, these functions are not continuous on $\hat{A}$. They are, however, always lower semi-continuous.

\begin{prop}[{\cite[Proposition 3.5.9]{D77}}]\label{lowersemicts}
    For any C*-algebra $A$, the functions $\{\Tr_a: a \in A^+\}$ are lower semi-continuous on $\hat{A}$.
\end{prop}

There are three equivalent definitions of continuous-trace C*-algebras. We will use the definition involving continuous-trace elements and $\Tr_a$.

\begin{defi}
    Let $A$ be a C*-algebra with Hausdorff spectrum. A positive element $a \in A^+$ is said to be a \emph{continuous-trace} element if $\Tr_a$ is a continuous function on $\hat{A}$.
\end{defi}

The span of the continuous-trace elements of a C*-algebra forms a two-sided (not necessarily closed) ideal $J$. The closure of this ideal then determines when an algebra has continuous-trace.

\begin{defi}
    A (generally non-unital) C*-algebra $A$ with Hausdorff spectrum is said to have \emph{continuous-trace} if the ideal $J$, generated by continuous-trace elements, is dense in $A$.
\end{defi}

\begin{rmk}\label{unitalctstrace}
    Observe that, given a unital C*-algebra $A$, $A$ has continuous-trace if and only if $1_A$ is a continuous-trace element. This is because the set of invertible elements $A^\times$ is an open subset of $A$, and so there is an invertible element in $J$ when $\clo{J}=A$. Furthermore, notice $\dim_A$ takes values in $\bN \cup \{\infty\}$. Therefore, $1_A$ is a continuous-trace element if and only if $\dim_A$ is constant on each component of $\hat{A}$.
\end{rmk}

A special class of continuous-trace C*-algebras is the class of homogeneous C*-algebras.
\begin{defi}
    A C*-algebra $A$ is \emph{homogeneous} if the dimension of each irreducible representation is the same natural number $n<\infty$.
\end{defi}

Homogeneous C*-algebras are automatically continuous-trace. When a homogeneous C*-algebra has compact spectrum, it is automatically unital (\cite[Theorem~3.2]{F61}). These will be important because unital continuous-trace C*-algebras are homogeneous on each connected component of their spectrum (see Remark~\ref{unitalctstrace}). We now present the Dixmier--Douady classification of continuous-trace C*-algebras.

\begin{thm}[{\cite[Theorem~5.29]{RW98}}]
    To each continuous-trace C*-algebra $A$ with Hausdorff spectrum $T$, there is an associated element $\delta(A)\in H^3(T;\bZ)$ called the Dixmier--Douady class of $A$. Two continuous-trace C*-algebras with spectrum $T$ are Morita equivalent over $T$ if and only if $\delta(A)=\delta(B)$.
\end{thm}
Equivalence classes of continuous-trace algebras form a group, where the group operation is given by the following relative tensor product. However, the assumption of continuous-trace is not necessary for the construction of this tensor product. We refer the reader to Appendix T in \cite{WO93} for the background on the max tensor product of C*-algebras.

\begin{defi}\label{tensorprodoverT}
    Given unital C*-algebras $A$ and $B$ with central $\ast$-homomorphisms $\phi \colon C(T)\to Z(A)$ and $\psi\colon C(T) \to Z(B)$, we write $A \otimes_T B$ for the C*-algebra $(A \otimes_\text{max}B)/I_T$, where $I_T$ is the balancing ideal generated by elements of the form
    $$a \phi(f)\otimes b - a \otimes \psi(f)b \qquad \text{ for } a \in A, b \in B, f \in C(T).$$
\end{defi}

Continuous-trace C*-algebras implicitly have such central $\ast$-homomorphisms due to the identifications of their spectra with $T$. This allows for the group of Morita equivalence classes of continuous-trace algebras to be defined as follows:

\begin{defi}[{\cite[Theorem~6.3]{RW98}}]
    Given a compact Hausdorff space $T$, define the \emph{Brauer group} of $T$, denoted $\text{Br}(T)$, to be the group whose elements are Morita equivalence classes (over $T$) of (non-unital) continuous-trace C*-algebras with spectrum $T$. The group operation is given by $[A][B]=[A\otimes_TB]$. The identity element is $[C(T)]$, and $[A]^{-1}=[A^\op]$.
\end{defi}

Considering non-unital continuous-trace C*-algebras is important for the Brauer group, as it allows for the Dixmier--Douady class to be a surjective map onto $H^3(T;\bZ)$. When restricting to unital continuous-trace algebras, the (unital) Brauer group surjects onto the torsion subgroup $H^3_{\tor}(T;\bZ)$ (see Lemma~\ref{dixmierdouady1mors}).

\begin{thm}[{\cite[Theorem~6.3]{RW98}}]\label{brauergroupiso}
    The map $\delta\colon \text{Br}(T)\to H^3(T;\bZ)$ given by $[A]\mapsto \delta(A)$ is a group isomorphism.
\end{thm}
We refer the reader to \cite{RW98} for a thorough exposition on continuous-trace algebras.
\subsection{\texorpdfstring{C*}{C*}-categories}
We now give the necessary background on C*-categories.
\begin{defi}
A C*-category $C$ consists of the following data.
\begin{itemize}
\item A category $C$. For objects $X,Y \in C$, we will denote the corresponding hom space by $C(X \to Y)$. We will use lowercase $f,g,\hdots$ for elements of $C(X \to Y)$ and use a contravariant $\circ$ to denote their composition.
\item A dagger structure $\dag$ on $C$. This consists of conjugate linear maps 
$$\dag \colon \overline{C(X \to Y)} \to C(Y \to X)$$
satisfying $(f \circ g)^\dag = g^\dag \circ f^\dag$ and $f^{\dag \dag} = f$.
\end{itemize}
We require that the dagger structure equips each endomorphism algebra $\End(X) \coloneqq C(X \to X)$ with the structure of a C*-algebra and that $f^\dag \circ f$ is positive in $\End(X)$ for $f \in C(X \to Y)$. In particular, each hom-space $C(X \to Y)$ is equipped with the structure of a Banach space through the norm
$$\|f \| \coloneqq \|f^\dag \circ f\|^{1/2}_{\End(X)} \qquad\qquad f \in C(X \to Y).$$
\end{defi}

\begin{defi}
A C*-2-category $\mathcal{C}$ consists of the following data.
\begin{itemize}
\item A 2-category $\mathcal{C}$. For objects $A,B \in \mathcal{C}$, we denote the corresponding hom category by $\mathcal{C}(A \to B)$. For $X,Y \in \mathcal{C}(A \Rightarrow B)$, we denote the corresponding hom space by $\mathcal{C}({}_AX_B \Rightarrow {}_AY_B)$. We will use lowercase $f,g,\hdots$ for elements of $\mathcal{C}({}_AX_B \Rightarrow {}_AY_B)$. Furthermore, we will use a contravariant $\circ$ for the composition of these 2-morphisms, whereas we will use a covariant $\otimes$ for the composition associated to 1-morphisms, i.e.,
$$
X \in \mathcal{C}(A \to B), \quad  Y \in \mathcal{C}(B \to C), \qquad X \otimes Y \in \mathcal{C}(A \to C).
$$
\item The structure of a C*-category on each hom category $\mathcal{C}({}_AX_B \Rightarrow {}_AY_B)$. 
\end{itemize}
We require that this data satisfies $(f \otimes g)^\dag = f^\dag \otimes g^\dag$ and that the associators and unitors associated to $\otimes$ are unitary. 
\end{defi}

\begin{defi}[{\cite{Ferrer24}}]
A C*-3-category $\mathfrak{C}$ consists of the following data.
\begin{itemize}
\item An algebraic tricategory $\mathfrak{C}$. For objects $S,T \in \mathfrak{C}$, we denote the corresponding hom 2-category by $\mathfrak{C}(S \to T)$. For $A,B \in \mathfrak{C}(S \to T)$, we denote the corresponding hom 2-category by $\mathfrak{C}({}_SA_T \Rightarrow {}_SB_T)$. Finally, for $X,Y \in  \mathfrak{C}({}_SA_T \Rightarrow {}_SA_T)$ we denote the corresponding hom category by $\mathfrak{C}({}_AX_B \Rrightarrow {}_AY_B)$. We will use lowercase $f,g,\hdots$ for elements of $ \mathfrak{C}({}_AX_B \Rrightarrow {}_AY_B)$. Furthermore, we will use a contravariant $\circ$ for the composition of these 3-morphisms, whereas we will use a covariant $\otimes$ and $\boxtimes $ for the compositions associated to 2-morphisms and 1-morphisms respectively. 
\item The structure of a C*-2-category on each 2-category $\mathfrak{C}(S \to T)$. 
\end{itemize}
We require that this data satisfies $(f \boxtimes g)^\dag = f^\dag \boxtimes g^\dag$ and that the higher coherence isomorphisms associated to $\boxtimes$ are unitary.
\end{defi}

The present work focuses on the following example of a C*-3-category as seen in Section~C of \cite{Ferrer24} (where it is denoted $\mathsf{CHaus}$).
\begin{defi}
There is a C*-3-category $\AbCAlg$ consisting of
\begin{enumerate}
\item[(0)] An object in $\AbCAlg$ is an abelian C*-algebra $C(T)$ with $T$ a compact Hausdorff space. We will often use $C(T)$ and $T$ interchangeably in our notation.
\item[(1)] A 1-morphism $A \in \AbCAlg(S \to T)$ is a unital (generally noncommutative) C*-algebra $A$ equipped with central maps $C(S) \xrightarrow{\phi} Z(A) \xleftarrow{\psi} C(T)$. We will often denote this data by ${}_\phi A _\psi$ when the source and target $C(S)$ and $C(T)$ are clear.
\item[(2)] A 2-morphism $X \in \AbCAlg({}_SA_T \Rightarrow {}_SB_T)$ is a C*-correspondence ${}_AX_B$ compatible with the central inclusions of $C(S)$ and $C(T)$ into $A$ and $B$.
\item[(3)] A 3-morphism $f \in \AbCAlg({}_AX_B \Rrightarrow {}_AY_B)$ is an adjointable $A$-$B$ bimodule map $f \colon X \to Y$.
\item[($\boxtimes$)] The composition at the level of 1-morphisms ${}_\phi A_\psi \in \AbCAlg(S \to T)$ and ${}_\mu B_\nu \in \AbCAlg(T \to U)$ is determined by the relative max tensor product $A \otimes_T B$ (Definition~\ref{tensorprodoverT}), which is the pushout of the following diagram:
\[
\begin{tikzcd}
C(T) \arrow[r,"\psi \otimes 1_B"] \arrow[d,"1_A \otimes \mu"'] & A \otimes_{\max} B \arrow[d,dashed] \\
A \otimes_{\max} B \arrow[r,dashed] & A \otimes_T B
\end{tikzcd}
\]
\item[($\otimes$)] Similarly, the composition at the level of 2-morphisms ${}_AX_B, {}_B Y_C$ is given by the relative max tensor product ${}_A X \otimes_B Y_C$. 
\item[($\circ$)] Composition at the level of 3-morphisms $f,g$ is simply given by function composition $g \circ f$.
\item[($\dag$)] The dagger of an adjointable bimodule map $f \colon {}_A X_B \to {}_A Y_B$ is simply given by its adjoint $f^\dag \colon {}_A Y_B \to {}_A X_B$.
\end{enumerate}
The necessary constraint data to promote $\AbCAlg$ into a C*-3-category is induced by maps obtained from the universal property the relative max-tensor product enjoys. We refer the interested reader to \cite{Ferrer24} for the remaining details.
\end{defi}
\begin{rmk}\label{moritainvertible}
    We note that if $X \colon A \Rightarrow B$ is an invertible $2$-morphism in $\AbCAlg$, then $X$ is actually an $A$-$B$ imprimitivity bimodule (and so is its inverse) by Lemma 2.4 in \cite{EKQR06}.
\end{rmk}
We now give the definition of the homotopy groups in a $3$-category.
\begin{defi}
    Let $\mathfrak{C}$ be a $3$-category and $T$ an object in $\mathfrak{C}$.
    \begin{enumerate}
        \item $\pi_0(\mathfrak{C})$ is defined to be the set $\text{Ob}(\mathfrak{C})$ up to the equivalence relation induced by the $1$-, $2$-, and $3$-morphisms.
        \item $\pi_1(\mathfrak{C},T)$ is the group of invertible $1$-morphisms $T\to T$ up to the equivalence relation induced by $2$- and $3$-morphisms.
        \item $\pi_2(\mathfrak{C}, T)$ is the group of invertible $2$-morphisms $\id_T \Rightarrow \id_T$ up to $3$-iso\-morphism.
        \item $\pi_3(\mathfrak{C}, T)$ is the group of invertible $3$-morphisms $\id_{\id_T}\Rrightarrow\id_{\id_T}$.
        
    \end{enumerate}
\end{defi}

We will leave the definition of the actions of $\pi_1$ on $\pi_2$ and $\pi_3$ for Section \ref{section:actions}.

\subsection{Duality for Abelian \texorpdfstring{C*}{C*}-Algebras}
Here we recount the important categorical equivalences between abelian C*-algebras and topology.

\begin{thm}[Gelfand duality]
    There is an equivalence of categories
    $$\mathsf{CHaus}^{\text{op}} \to \mathsf{AbC^*Alg_1}$$
    from the opposite category of compact Hausdorff spaces and continuous maps to the category of unital, abelian C*-algebras with unital $\ast$-homomorphisms.
\end{thm}
The equivalence is witnessed by sending a compact Hausdorff space $T$ to the C*-algebra $C(T)$, and a continuous map $\Phi \colon T \to S$ is sent to the $\ast$-homomorphism $\Phi^*\colon C(S) \to C(T)$ given by precomposition with $\Phi$. Throughout this work, we will use $\Phi^*$ to denote the $\ast$-homomorphism given by this equivalence. Conversely, given a unital $\ast$-homomorphism $\phi\colon C(T) \to C(S)$, we will use $\hat\phi\colon S \to T$ to denote its preimage under this equivalence.

\par The Serre--Swan Theorem relates modules over abelian C*-algebras to vector bundles over the underlying topological space. We briefly recount the definition of a complex vector bundle.

\begin{defi}
    A \emph{complex vector bundle} over a compact Hausdorff space $T$ is a topological space $E$ (the total space) along with a continuous surjection $p\colon E \to T$ such that the fibers $p^{-1}(\{t\})$ are vector spaces for all $t \in T$. Furthermore, the vector bundle $E$ is asked to be locally trivial in the following sense: for each $t \in T$, there is an open set $U\subseteq T$ containing $t$ such that $p^{-1}(U)\cong U \times \bC^k$ (for some $k$). This isomorphism is required to be the identity on $U$ (by respecting $p$) and is a linear isomorphism when restricted to the fibers of $p^{-1}(U)$.
\end{defi}

Because we want to obtain Hilbert C*-modules from vector bundles instead of simply $C(T)$-modules, we need our vector bundles to be equipped with Hermitian metrics.

\begin{defi}
    A \emph{Hermitian metric} on $E$ is a continuous function $\langle \cdot | \cdot \rangle_E \colon \clo{E}\otimes E \to \bC$ that restricts to an inner product on each fiber $p^{-1}(\{t\}) \times p^{-1}(\{t\})$.
\end{defi}

A module is obtained from a vector bundle by considering its continuous sections.

\begin{defi}
    A (global) \emph{continuous section} on a vector bundle $E$ over $T$ is a continuous function $f \colon T \to E$ such that $p \circ f= \id_T$. The set of all continuous sections is denoted $\Gamma(E)$. Evidently, this is a module over $C(T)$ given by pointwise multiplication (using the vector space structure on each fiber).
\end{defi}

We now have all of the necessary background to state the Serre--Swan Theorem for abelian C*-algebras.

\begin{thm}[{\cite{S62}}]\label{thm:SerreSwan}
    There is an equivalence of C*-categories
    $$\mathsf{{Hilb}_{fd}}(T) \to\mathsf{Hilb_{fgp}}C(T)$$
    from the category of finite-rank topological hermitian vector bundles over $T$ to the category of finitely generated, projective Hilbert $C(T)$-modules. 
\end{thm}

The equivalence is witnessed by sending a vector bundle $E$ to its space of continuous global sections $\Gamma(E)$. The $C(T)$-valued inner product is given by 
$$\langle f | g \rangle_{C(T)} (t)=\langle f(t) | g(t)\rangle_E$$ for the Hermitian metric $\langle \cdot | \cdot \rangle_E$. A map of vector bundles $\tau \colon E \to F$ is sent to the map $\Gamma(\tau)\colon \Gamma(E) \to \Gamma(F)$ defined by $\Gamma(\tau)(f)=\tau \circ f$ for all $f \in \Gamma(E)$.
\par It is known to experts that Swan's theorem is actually a monoidal equivalence of monoidal categories (with the usual tensor product of vector bundles and relative tensor product of modules). This fact will be required in subsection~\ref{subsec:pi2}. As we have been unable to find a reference, we provide a proof in Appendix \ref{Sec:Serre-Swan}.
\par We will also be interested in the case when $T$ has the homotopy type of a CW-complex, as isomorphism classes of line bundles are determined by a cocycle in $H^2(T;\bZ)$. We use $\operatorname{Pic}(T)$ to denote the group of isomorphism classes of (complex) line bundles over $T$.
\begin{prop}[{\cite[Proposition 3.10]{H03}}]
    When $T$ has the homotopy type of a CW-complex, the isomorphism class of a complex line bundle $E$ is determined by its first Chern class $c_1(E)$ in $H^2(T;\bZ)$. The map $E \mapsto c_1(E)$ is a group isomorphism from $\operatorname{Pic}(T) \to H^2(T;\bZ)$.
\end{prop}

\section{Computing the homotopy groups of \texorpdfstring{$\AbCAlg$}{AbC*Alg}}\label{section:homotopygroups}
In this section, we will compute the homotopy groups in $\AbCAlg$. Some of the arguments are easier to understand using $\ast$-isomorphisms rather than invertible bimodules. The following lemma justifies using appropriate $\ast$-isomorphisms in place of $2$-morphisms in $\AbCAlg$.
\begin{lem}\label{equivariantstariso}
    Let ${_\phi}A_\psi$ and ${_\mu}B_\nu$ be $1$-morphisms from $C(T) \to C(S)$. Suppose there is a $\ast$-isomorphism $\tau\colon A \to B$ such that $\tau \circ \phi = \mu$ and $\tau \circ \psi = \nu$. Then $\tau$ defines a $2$-isomorphism $A \Rightarrow B$ in $\AbCAlg$.
\end{lem}

\begin{proof}
    Consider the bimodule ${_A}B_B$ with the usual $B$-action $(b_1 \triangleleft b_2= b_1b_2)$ and $B$-valued inner product $(\langle b_1 | b_2 \rangle = b_1^\dagger b_2)$, and with the left action by $A$
    $$a \triangleright b = \tau(a)b.$$
    It is routine to verify that ${_A}B_B$ is an $A$-$B$ imprimitivity bimodule, with $A$-valued inner product $$\langle b_1, b_2\rangle = \tau^{-1}(b_1b_2^\dagger).$$ To see that $_{A}B_B$ is a $2$-morphism in $\AbCAlg$ (and therefore a $2$-equivalence), observe that for $f \in C(T)$ and $b \in B$, we have
    $$\phi(f)\triangleright b = (\tau\circ \phi)(f)b= \mu(f)b=b \triangleleft \mu(f).$$
    We similarly see that the left and right actions by continuous functions on $S$ agree, and so ${_A}B_B$ is a $2$-equivalence in $\AbCAlg$.
\end{proof}

\begin{rmk}
    In general, we cannot replace a $2$-equivalence by a $\ast$-isomorphism. For example, $M_2(\bC)$ and $\bC$ are not isomorphic C*-algebras, but are Morita equivalent, and so produce equivalent $1$-morphisms in $\AbCAlg$ (assuming the central maps are chosen appropriately).
\end{rmk}
Our first major goal is the construction of $\pi_1(\AbCAlg, C(T))$. This can be broken into two parts: describing the structure of the central $\ast$-homomorphisms $\phi,\psi \colon C(T) \to Z(A)$ and the C*-algebraic structure of $A$. In the following subsection, we will work towards understanding the central $\ast$-homomorphisms. Along the way, we pick up a characterization of $\pi_0(\AbCAlg)$.
\subsection{Results for \texorpdfstring{$\pi_0$}{the zeroth} and \texorpdfstring{$\pi_1$}{the first homotopy groups}}\label{subsec:pi0}
Proving that the central $\ast$-homomorphisms are actually $\ast$-isomorphisms is proven using the faithfulness of the actions on imprimitivity bimodules.

\begin{lem}\label{centraliso}
If ${}_{\phi}A_{\psi}$ is a $1$-isomorphism $C(T)\to C(S)$, both $\phi \colon C(T) \to Z(A)$ and $\psi \colon C(S) \to Z(A)$ are $*$-isomorphisms. In particular, $T\cong S$.
\end{lem}

\begin{proof}
    We begin with the case where $T=S$ and we have a $1$-morphism ${_\eta}C_\kappa$ that is equivalent to ${}_{\id}C(T)_{\id}$ via a $2$-isomorphism ${_C}X_{C(T)}$. By Remark~\ref{moritainvertible}, ${_C}X_{C(T)}$ is an imprimitivity bimodule. Since ${_C}X_{C(T)}$ is full as a left Hilbert module, the left action of $C$ is faithful by Remark~\ref{fullimpliesfaithful}.
    \par We first show that $\eta$ is a $\ast$-isomorphism onto $Z(C)$. We begin by showing injectivity. Suppose $f\in C(T)$ satisfies $\eta(f)=0$. For $x \in X$, we have $x \triangleleft f = \eta(f) \triangleright x = 0$ since ${_C}X_{C(T)}$ is a $2$-morphism. As the right action of $C(T)$ is faithful, we conclude that $f=0$, so $\eta$ is injective. To show surjectivity, suppose $c \in Z(C)$. Then we may find a function $g \in C(\Prim(C))$ so that $\cF(g)=c$, where $\cF\colon C(\Prim(C)) \to Z(C)$ is the isomorphism given by the Dauns--Hofmann Theorem (Theorem~\ref{DHthm}). As ${_C}X_{C(T)}$ is an imprimitivity bimodule, it induces a Rieffel homeomorphism $h_X\colon T \to \Prim(C)$. We then have, for all $x \in X$,
    $$c\triangleright x = \cF(g)\triangleright x =x \triangleleft (g \circ h_X) = \eta(g \circ h_X) \triangleright x$$
    by Proposition~\ref{rieffelhomeoactions}, where we treated the Dauns--Hofmann isomorphism $\mathcal{F}$ as the identity on $C(T)$. Since the left $C$ action is faithful, we conclude that $c=\eta(g \circ h_X)$, and so $\eta$ is an isomorphism $C(T) \to Z(C)$. A similar argument shows that $\kappa\colon C(T) \to Z(C)$ is an isomorphism as well.
   \par For the general case, suppose ${_\mu}B_\nu$ is an inverse of ${_\phi}A_\psi$. Applying the above argument to the $1$-morphism ${_\eta}C_\kappa={_\phi}(A \otimes_S B)_\nu$ proves that $\phi\otimes 1_B$ is a $\ast$-isomorphism onto $Z(A\otimes_S B)$. However, 
   $$\phi(C(T))\otimes_S 1_B \subseteq Z(A)\otimes_S 1_B \subseteq Z(A \otimes_S B)$$
   from which we conclude that $\phi$ is a $\ast$-isomorphism $C(T) \to Z(A)$ as
   \[Z(A) \cong Z(A) \otimes_S C(S) \cong Z(A) \otimes_S 1_B \subseteq Z(A\otimes_S B)\]
   Similarly, applying the argument to ${_\mu}(B \otimes_T A)_\psi$ proves that $\psi\colon C(S) \to Z(A)$ is a $\ast$-isomorphism. We then see that $\psi^{-1}\circ \phi\colon C(T)\to C(S)$ is a $\ast$-isomorphism, from which we conclude that $T \cong S$ by Gelfand duality.
\end{proof}

The characterization of $\pi_0$ follows immediately from the preceding lemma.
\begin{thm} 
The collection of equivalence classes of objects in $\AbCAlg$ is given by
$$\pi_0(\AbCAlg) = \pi_0(\mathsf{CHaus}) = \{ \text{compact Hausdorff spaces}\}/\text{homeomorphism}.$$ 
\end{thm}

We now continue with our analysis of the central $\ast$-iso\-morphisms for elements $\pi_1(\AbCAlg, C(T))$. We will often use the following lemma about the center of $1$-morphisms. It also gives a suggestion of the semidirect product decomposition of $\pi_1(\AbCAlg, C(T))$.
\begin{prop}\label{prop:splitoffcenter} 
    Let $T$ be a compact Hausdorff space and ${_\phi}A_\psi$ a $1$-automorphism of $C(T)$. Then ${_\phi}A_\psi \cong {_\phi}Z(A)_\psi \otimes_T {_\psi}A_\psi \cong {_{\psi^{-1} \circ \phi}}C(T)_{\id} \otimes_T {_\psi}A_\psi$. When $A=C(T)$, we have $_\phi C(T)_\psi \cong {_{\psi^{-1} \circ \phi}}C(T)_{\id}$.
\end{prop}

\begin{proof}
    It is routine to see that $_\phi A_\psi$ and $_\phi Z(A)_\psi \otimes_T {_\psi}A_\psi$ are isomorphic as C*-algebras via the map $a \mapsto 1\otimes_T a$. It is clear that this $\ast$-isomorphism preserves the central maps $\phi$ and $\psi$. Thus, they are equivalent $1$-morphisms in $\AbCAlg$ by Lemma~\ref{equivariantstariso}. Now, considering the $\ast$-isomorphism $\psi\colon {_{\psi^{-1} \circ \phi}}C(T)_{\id} \to {_\phi}Z(A)_\psi$, we see that this is another isomorphism of C*-algebras that respects the central maps $\psi$ (resp. $\phi$) and $\id$ (resp. $\psi^{-1} \circ \phi$). Once again, Lemma~\ref{equivariantstariso} proves they are equivalent $1$-morphisms in $\AbCAlg$. The special case when $A=C(T)$ immediately follows.
\end{proof}

When $A=C(T)$, we see that the central $\ast$-homomorphisms can be moved to one side of the $1$-morphism. This allows us to construct a map from ${\homeo}(T)$ to $\pi_1(\AbCAlg, C(T))$. The following proposition describes how these morphisms compose and proves they form a subgroup.

\begin{prop}[Arithmetic with Homeomorphisms]\label{homeoarithmetic} 
    Let $T$ be a compact Hausdorff space, and let $\phi$ and $\psi$ be automorphisms of $C(T)$. Then the $1$-automorphisms $_{\phi}C(T)_{\id} \otimes_T{_{\psi}}C(T)_{\id}$ and $_{\psi \phi}C(T)_{\id}$ are equivalent. Furthermore, $_{\phi}C(T)_{\id}$ and $_{\psi}C(T)_{\id}$ are equivalent if and only if $\phi=\psi$. Therefore, the map $\Pi\colon {\homeo}(T)\to \pi_1(\AbCAlg, C(T)))$ given by $\Phi \mapsto {_{\Phi^*}}C(T)_{\id}$ is an injective group homomorphism.
\end{prop}

\begin{proof}
    Define a $\ast$-isomorphism $\rho\colon {_\phi}C(T)_{\id} \otimes_T {_\psi} C(T)_{\id} \to {_{\psi \circ \phi}}C(T)_{\id}$ given by
    $$f \otimes_T g \mapsto \psi(f)g.$$
    Note that this respects the central maps, in that $\rho \circ (\phi \otimes_T 1)= \psi\circ \phi$ and $\rho \circ (1 \otimes_T \id) = \id$. Thus these are equivalent $1$-morphisms in $\AbCAlg$ by Lemma~\ref{equivariantstariso}.
    \par Now, suppose ${_{\phi}}C(T)_{\id}$ and ${_{\psi}}C(T)_{\id}$ are equivalent; that is, there is a $C(T)$-$C(T)$ imprimitivity bimodule $X$ satisfying the properties
    $$\textbf{(a): } \phi(f)\triangleright x = x \triangleleft \psi(f) \text{ and }\textbf{(b): }\id(f)\triangleright x = x \triangleleft \id(f)$$
    for all $x \in X$ and $f \in C(T)$. Then, in particular,
    $$\phi(f)\triangleright x \overset{\textbf{(a)}}{=} x\triangleleft \psi(f)\overset{\textbf{(b)}}{=}\psi(f)\triangleright x. $$
    As this holds for all $x \in X$, and $X$ is full, it follows that $\phi(f)=\psi(f)$, from which we conclude that $\phi=\psi$.
    \par To show that $\Pi$ is a group homomorphism from $\text{Homeo(T)}$  to $\pi_1(\AbCAlg, C(T))$, let $\Phi$ and $\Psi$ belong to $\text{Homeo(T)}$. Then
    \begin{align*}
        \Pi(\Phi)\otimes_T\Pi(\Psi)&= {_{\Phi^*}}C(T)_{\id}\otimes_T{_{\Psi^*}}C(T)_{\id}\\
        &\cong {_{\Psi^*\circ \Phi^*}}C(T)_{\id}\\
        &={_{(\Phi\circ \Psi)^*}}C(T)_{\id}\\
        &= \Pi(\Phi\circ \Psi).
    \end{align*}
    We conclude that $\Pi$ is an injective group homomorphism.
\end{proof}

Having obtained a good handle on the central $\ast$-homomorphisms, we now construct a correspondence between spectra. Because we have maps from $C(T)$ to $Z(A)$ rather than from $T$ to $\hat{A}$, we need to formally understand the connection between $\hat{A}$ and $Z(A)$. The following result is well-known, but we will need an explicit description of the restriction map in our setting.
\begin{lem}\label{restrict}
    Let $\lambda \in \hat{A}$. Then the $\ast$-homomorphism $\restrict{\lambda}{Z(A)}\colon Z(A) \to \mathbb{C}$ belongs to $\widehat{Z(A})$, and the map $\operatorname{res}\colon \hat{A} \to \widehat{Z(A)}$ given by $\lambda \mapsto \restrict{\lambda}{Z(A)}$ is a continuous surjection. This map is injective if and only if $\hat{A}$ is Hausdorff.
\end{lem}

\begin{proof}
    By Lemma~1.1 in \cite{N95}, $\restrict{\lambda}{Z(A)}$ is in $\widehat{Z(A)}$. Furthermore, the map $\ker(\lambda)\mapsto \ker(\restrict{\lambda}{Z(A)})$ is a continuous map from $\Prim{A}$ to $\Prim{Z(A)}$ with dense range. Since $A$ is unital, $\Prim{A}$ is compact by Proposition~\ref{prop:spectrumcompact}, so this map is actually onto. Thus, $\lambda \mapsto \ker(\restrict{\lambda}{Z(A)})$ is a continuous surjection $q$ of $\hat{A}$ onto $\Prim{Z(A)}$. As $Z(A)$ is abelian, $\ker\colon \widehat{Z(A)}\to \Prim{Z(A)}$ is a homeomorphism, so we have $\ker^{-1} \circ q\colon \hat{A} \to \widehat{Z(A)}$ is a continuous surjection and, in particular, $(\ker^{-1}\circ q)(\lambda)=\restrict{\lambda}{Z(A)}$. The argument is summarized in the below diagrams.

    \[
    \begin{tikzcd}[column sep=30pt]
    \widehat{A} \arrow[r,"\text{res}"] \arrow[dr,"q"description]\arrow[d,"\ker"'] & \widehat{Z(A)} \arrow[d,"\ker"] \\
    \Prim{A} \arrow[r,"- \cap Z(A)"'] & \Prim{Z(A)}
    \end{tikzcd}  
    \qquad
    \begin{tikzcd}[column sep=30pt]
    \lambda \arrow[r,mapsto] \arrow[dr,mapsto]\arrow[d,mapsto] & \lambda|_{Z(A)} \arrow[d,mapsto] \\
    \ker(\lambda) \arrow[r,mapsto] & \ker(\lambda|_{Z(A)})
    \end{tikzcd}  
    \]
    \par If the restriction map is injective, then it is a continuous bijection from a compact space to a Hausdorff space and is therefore a homeomorphism, from which it follows that $\hat{A}$ is Hausdorff. Conversely, if $\hat{A}$ is Hausdorff, then $\ker\colon \hat{A} \to \Prim{A}$ must be injective because the topology on $\hat{A}$ is pulled back from the topology on $\Prim{A}$. But, then $\ker\colon \hat{A} \to \Prim{A}$ is a homeomorphism, and so $\Prim{A}$ is Hausdorff. Thus, as $\Prim{A}$ is also compact, the map $-\cap Z(A)$ is a homeomorphism by Theorem~2.2 of \cite{N95}. It follows that $\text{res}$ is a homeomorphism and, in particular, is injective.
\end{proof}

A $1$-morphism composed with its inverse will be Morita equivalent to $C(T)$, so there is much that can be said about the structure of this composition. In particular, the spectrum of the relative tensor product ${A\otimes_T A^{-1}}$ will be homeomorphic to $T$. We want to use this fact to construct a homeomorphism of $\hat{A}$ and $T$. As a first step, the following remark relates the spectrum of $A$ to the spectrum of $A \otimes_T A^{-1}$.

\begin{rmk}\label{spectensorprod}
    Let ${_\phi}A_\psi$ and ${_\mu}B_\nu$ be two $1$-automorphisms of $C(T)$. Observe that the pairs $(\lambda,\rho)\subset \hat{A} \times \hat{B}$ satisfying $\restrict{\lambda}{Z(A)} \circ \psi = \restrict{\rho}{Z(B)} \circ \mu$ lie in the spectrum of $A\otimes_T B$. For, we have an inclusion of $\hat{A} \times \hat{B}$ into the spectrum of $A\otimes_{\text{max}}B$ by sending $(\lambda,\rho)\mapsto \lambda \otimes \rho$ (This map is a homeomorphism onto its range in the spectrum of $A\otimes_{\text{min}}B$ by \cite[Theorem B.45]{RW98}, and $A\otimes_{\text{max}}B$ surjects onto $A\otimes_{\text{min}}B$). We then see that the representations $\lambda \otimes \rho$ satisfying $\restrict{\lambda}{Z(A)} \circ \psi = \restrict{\rho}{Z(B)} \circ \mu$ are precisely the representations of this form that restrict to $0$ on the balancing ideal $I_T$, and so are in the spectrum of $A\otimes_T B$.
\end{rmk}

Having established the preceding relationship, we can now work towards constructing the desired homeomorphism $\hat{A} \to T$. As an intermediate step, we will prove that the spectrum is Hausdorff, which will subsequently allow us to apply Lemma \ref{restrict}.

\begin{lem}\label{t2spec}
        Suppose that $_\phi A_\psi$ is a $1$-automorphism. Then the spectrum $\hat{A}$ is Hausdorff.
    \end{lem}

    \begin{proof}
        By Lemma~\ref{restrict}, it suffices to show that the restriction map $[\lambda]\mapsto [\restrict{\lambda}{Z(A)}]$ is injective. Let ${_\mu}B_\nu$ be $({_\phi}A_\psi)^{-1}$, so that ${_\phi}(A \otimes _T B)_\nu$ is Morita equivalent to ${_{\id}}C(T)_{\id}$. To show that it is injective, we consider two representations $\lambda_1$ and $\lambda_2$ in $\hat{A}$ and assume $\operatorname{res}(\lambda_1)=\operatorname{res}(\lambda_2)$; that is, $\restrict{\lambda_1}{Z(A)}$ and $\restrict{\lambda_2}{Z(A)}$ are unitarily equivalent. As $Z(A)$ is an abelian C*-algebra, this happens if and only if $\restrict{\lambda_1}{Z(A)}=\restrict{\lambda_2}{Z(A)}$ as maps into $\bC$. As $Z(B)\cong C(T)$ by Lemma~\ref{centraliso}, we may choose $\rho\in \hat{B}$ so that $\restrict{\rho}{Z(B)}=\restrict{\lambda_1}{Z(A)} \circ (\psi \circ \mu^{-1}) \in \widehat{Z(B)}$ by Lemma~\ref{restrict}. We then see that $\lambda_i \otimes \rho$ is an element of $\widehat{A\otimes_T B}$ for $i=1,2$ by Remark~\ref{spectensorprod}. Furthermore, we claim $\lambda_1 \otimes_T \rho$ and $\lambda_2 \otimes_T \rho$ are unitarily equivalent irreducible representations. For, since ${_\phi}(A\otimes_T B)_\nu$ is Morita equivalent to $C(T)$, the spectrum of ${_\phi}(A\otimes_T B)_\nu$ is homeomorphic to $T$, and, in particular, is Hausdorff by Theorem~\ref{moritaspectra}. Then, by Lemma~\ref{restrict}, the restriction map is a homeomorphism. Since the restrictions of $\lambda_i\otimes_T \rho$ are equal, they are unitarily equivalent. However, we then have $\lambda_1 \otimes_T \rho(1_B)$ and $\lambda_2 \otimes_T \rho(1_B)$ are unitarily equivalent, and it follows that $\lambda_1$ and $\lambda_2$ are unitarily equivalent by \cite[Lemma B.47]{RW98}. We conclude that $[\lambda_1]=[\lambda_2]$ in $\hat{A}$, so that the restriction map is injective.
    \end{proof}

We now construct homeomorphisms $\hat{A}\to T$ from the central $\ast$-isomorphisms.

\begin{prop}\label{spectrumof1iso}
    If $_\phi A _\psi$ is a 1-automorphism of $C(T)$, then $\hat{A} \cong T$. In particular, $\hat{\phi}, \hat{\psi}\colon \hat{A}\to T$ given by $\hat{\phi}(\lambda)=\restrict{\lambda}{Z(A)} \circ \phi$ and $\hat{\psi}(\lambda)=\restrict{\lambda}{Z(A)} \circ \psi$ are homeomorphisms. 
\end{prop}

\begin{proof}
By Lemma \ref{t2spec}, $\hat{A}$ is Hausdorff, and hence $\hat{A} \cong \widehat{Z(A)}$ by Lemma \ref{restrict}. On the other hand, $Z(A) \cong C(T)$ by Lemma \ref{centraliso}, so we conclude $\hat{A} \cong \widehat{Z(A)} \cong T$. To describe these isomorphisms more explicitly, recall that the map from $\hat{A} \to \widehat{Z(A)}$ is given by sending $\lambda \mapsto \restrict{\lambda}{Z(A)}$. Furthermore, our identification of $Z(A)$ with $C(T)$ is given by $\phi$, and so the identification $\widehat{Z(A)} \to \widehat{C(T)}= T$ is given by $\rho \mapsto \rho \circ \phi$. Composing these maps yields a map $\lambda \mapsto \restrict{\lambda}{Z(A)} \circ \phi$, which is the map $\hat\phi$. By a similar argument, $\hat\psi$ is a homeomorphism as well.
\end{proof}

Having proved each morphism $_{\phi}A_\psi$ in $\pi_1(\AbCAlg, C(T))$ has spectrum $\hat{A}\cong T$, we can now show that the C*-algebra $A$ has continuous-trace.

\begin{thm}\label{1morctstrace}
	    If $_\phi A_\psi$ is a $1$-automorphism of $C(T)$, then $A$ is a continuous-trace algebra over $T$. 
	\end{thm}

    \begin{proof}
        By Proposition~\ref{spectrumof1iso}, we know that $\hat{A}\cong T$. Then, if ${_\mu}B_\nu$ is an inverse for ${_\phi}A_\psi$, we know that, in particular, $A\otimes_TB$ is Morita equivalent to $C(T)$. Thus $A\otimes_TB$ is a continuous-trace C*-algebra by Proposition~5.15 of \cite{RW98}. Now, define the set
        $$\Delta\coloneq\{\lambda \otimes \rho: (\lambda,\rho)\in \hat{A}\times\hat{B}, \lambda \circ \psi = \rho \circ \mu\}.$$
        Then the map $\Tr_{1_A \otimes 1_B}$ on $\widehat{A\otimes_TB}$ restricts to a continuous function on $\Delta$. Define $\gamma\colon \hat{A}\to\hat{B}$ by $\gamma=\hat{\mu}^{-1}\circ \hat{\psi}$. Observe that $\lambda \mapsto (\lambda \otimes \gamma(\lambda))$ is a homeomorphism $\hat{A}\cong \Delta$. As $A \otimes_TB$ is unital, we know this means that $1_A \otimes_T 1_B$ is a continuous-trace element by Remark \ref{unitalctstrace}. It follows that $\dim_{A \otimes _T B}$ is continuous on $\widehat{A \otimes_T B}$ and, in particular, takes finite values in $\bN$. Importantly, since $\dim_{A\otimes_T B}=\dim_A \cdot \dim_B$, we conclude that both $\dim_A$ and $\dim_B$ take on finite values. Thus,
        $$\dim_A(\lambda)=\frac{\dim_A(\lambda)\dim_B(\gamma(\lambda))}{\dim_B(\gamma(\lambda))}=\underbrace{\dim_{A\otimes_T B}(\lambda \otimes_T \gamma(\lambda))}_{\text{continuous}}\!\underbrace{\dim_B(\gamma(\lambda))^{-1}}_{\text{upper semicontinuous}}$$
        where upper semicontinuity of $[\lambda] \mapsto\dim_B(\gamma(\lambda))^{-1}$ comes from the fact that $\dim_B$ is lower semicontinuous and nonzero by Theorem~\ref{lowersemicts}. Therefore, $\dim_A$ is the product of two positive upper semicontinuous functions and is therefore upper semicontinuous. However, it is always lower semicontinuous (again by Theorem~\ref{lowersemicts}), so we conclude that $\dim_A$ is continuous. This means that $1_A$ is a continuous-trace element of $A$, and so $A$ has continuous-trace by Remark~\ref{unitalctstrace}.
    \end{proof}

We now wish to build a group homomorphism from $H^3_{\tor}(T;\bZ)$ to $\pi_1(\AbCAlg, T)$. Some care must be taken in constructing the map because our $1$-morphisms are constructed using unital C*-algebras.

\begin{lem}\label{dixmierdouady1mors}
    If ${_\phi}A_\phi$ is an invertible $1$-morphism of $C(T)$, then its Dixmier--Douady class $\delta(A)$ lies in $H^3_{\tor}(T;\bZ)$. Conversely, given any $\delta \in H^3_{\tor}(T;\bZ)$, there is a $1$-morphism ${_\psi}B_\psi$ of $C(T)$ with $\delta(B)=\delta$.
\end{lem}

\begin{proof}
We begin by showing that the Dixmier--Douady invariant is a torsion element of $H^3(T;\bZ)$. Given a compact Hausdorff space $T$, we may write $T$ as a disjoint union $T=\sqcup_{i=1}^nT_i$ for some connected components $T_i$ (finitely many because $T$ is compact). Thus, since the C*-algebra $A$ in ${_\phi}A_\phi$ is a unital continuous-trace algebra over $T$ by Theorem~\ref{1morctstrace}, we may write $A$ as the direct sum
$$A=\bigoplus_{i=1}^n\phi(\chi_{T_i})A.$$
Observe that $A_i\coloneq \phi(\chi_{T_i})A$ is a continuous-trace C*-algebra over $T_i$. Because $T_i$ is connected, the function $\dim_{A_i}$ must be constant, and so $A_i$ is a homogeneous C*-algebra. Thus, by Theorem IV.1.7.23 of \cite{B06}, the Dixmier--Douady class $\delta(A_i)$ must be a torsion element of $H^3(T_i;\bZ)$. Since cohomology respects direct sums, $\delta(A)=\sum_{i=1}^n\delta(A_i)$ must be a torsion element in $H^3(T;\bZ)$ as well.
\par For the reverse direction, given a torsion element $\delta \in H^3(T;\bZ)$, by \cite[Corollaries 1.5 \& 1.7]{G95} (c.f. \cite[Theorem~IV.1.7.24]{B06}), there is a homogeneous C*-algebra $B$ whose spectrum is identified with $T$ and whose Dixmer--Douady class is $\delta(B)=\delta$. This identification of $T$ with $\hat{B}$ corresponds to a $\ast$-isomorphism $\psi \colon C(T) \to Z(B)$ by the Dauns--Hofmann Theorem (Theorem~\ref{DHthm}). Furthermore, $B$ is a unital C*-algebra by Theorem 3.2 of \cite{F61}. Therefore, ${_\psi}B_\psi$ defines an invertible $1$-morphism $C(T) \to C(T)$ with $\delta(B)=\delta$.
\end{proof}

Because the Dixmier--Douady classification assumes a single identification $\hat{A} \cong T$ and our $1$-morphisms have two such identifications, we need to ensure that the morphisms with identical central maps genuinely form a subgroup of $\pi_1(\AbCAlg, C(T))$. The first part of the proof will also be important for proving the exactness of the split exact sequence we construct in Theorem~\ref{splitexactseq}.

\begin{lem}\label{singlecentralmap}
    The set of equivalence classes of $1$-morphisms of the form ${_\phi}A_\phi$ is a subgroup of $\pi_1(\AbCAlg,T)$.
\end{lem}

\begin{proof}
    We will first show that the equivalence classes of this form only consist of elements where both $\ast$-homomorphisms $C(T) \to Z(B)$ are equal. Suppose ${_\phi}A_\phi$ and ${_\mu}B_\nu$ are equivalent via an imprimitivity bimodule ${_A}X_B$. Because this defines a $2$-equivalence in $\AbCAlg$, $X$ satisfies the properties
    $$\textbf{(a): } \phi(f)\triangleright x = x \triangleleft \mu(f) \text{ and }\textbf{(b): }\phi(f)\triangleright x = x \triangleleft \nu(f)$$
    for all $x \in X$ and $f \in C(T)$. Then, in particular,
    $$x \triangleleft \mu(f) \overset{\textbf{(a)}}{=} \phi(f)\triangleright x \overset{\textbf{(b)}}{=} x \triangleleft \nu(f) .$$
    As $X$ is full as a Hilbert $B$-module, and the above holds for all $x \in X$ and $f \in C(T)$, we conclude that $\nu=\mu$ by Remark~\ref{fullimpliesfaithful}. So any representative of these equivalence classes will have a pair of identical central $\ast$-homomorphisms.
    \par Now, to show that this is a subgroup, first note that the set contains the identity element ${_{\id}}C(T)_{\id}$. Consider two $1$-automorphisms ${_\phi}A_\phi$ and ${_\psi}B_\psi$ of $C(T)$. Then their composite is ${_\phi}(A\otimes_TB)_\psi$, which will be in the purported subgroup if and only if the left and right maps from $C(T)\to Z(A\otimes_TB)$ are equal. However, for $f \in C(T)$, we have
    $$\phi(f)\otimes_T1_B=1_A\otimes_T \psi(f).$$
    Therefore, the composite may be more formally written
    $${_{\phi\otimes1_B}}(A\otimes_TB)_{1_A\otimes\psi}={_{\phi\otimes1_B}}(A\otimes_TB)_{\phi\otimes1_B}.$$
    We conclude that the set forms a subgroup of $\pi_1(\AbCAlg, T)$.
\end{proof}

We now have all of the necessary background to decompose $\pi_1(\AbCAlg, C(T))$ as a semidirect product.

\begin{thm}\label{splitexactseq}
There is a split exact sequence
\[
\begin{tikzcd}
0 \arrow[r] & H^3_{\operatorname{tor}}(T;\bZ) \arrow[r,"\ref{sses1}"] & \pi_1(\AbCAlg,T) \arrow[r,"\ref{sses2}"] & \arrow[l, bend right,dashed,"\ref{sses3}"'] \operatorname{Homeo}(T) \arrow[r] & 0.
\end{tikzcd}
\]
In particular, $\pi_1(\AbCAlg, T) = H^3_{\operatorname{tor}}(T;\bZ) \rtimes \operatorname{Homeo}(T).$
\end{thm}

\begin{proof}
We describe the aforementioned split short exact sequence as follows:
\begin{enumerate}
\item \label{sses1} Given a Dixmier-Douady class $\delta \in H^3_{\operatorname{tor}}(T;\bZ)$, there is a corresponding unital continuous-trace C*-algebra $A^\delta$ with a $\ast$-isomorphism $\psi\colon C(T) \to Z(A^\delta)$ by Lemma~\ref{dixmierdouady1mors}. Notice ${}_{\id}A^\delta_{\id} \in \pi_1(\AbCAlg,T)$ as the Dixmier-Douady class of $A^\delta \otimes_T (A^{\delta})^\text{op}$ is $0 \in H^3_{\operatorname{tor}}(T;Z)$; hence, $A^\delta \otimes_T (A^{\delta})^\text{op}$ is equivalent to ${_{\id}}C(T)_{\id}$ in $\AbCAlg$, where we consider $(A^{\delta})^\text{op}$ as a $1$-morphism with the same central $\ast$-homomorphisms as $A^\delta$. Note that as any two continuous-trace C*-algebras with the same Dixmier--Douady class are Morita equivalent over $C(T)$, the above map does not depend upon the choice of $A^\delta$. Furthermore, as the C*-algebras $A^\delta$ have an identical pair of central maps $C(T) \to Z(A)$, the $1$-composition of these morphisms agrees with the composition in the Brauer group. Thus, this defines a genuine group homomorphism $H^3_{\tor}(T;\bZ) \to \pi_1(\AbCAlg, T)$, which is injective by Theorem~\ref{brauergroupiso}.

\item \label{sses2} Define $\Lambda\colon \pi_1(\AbCAlg, T) \to {\homeo}(T)$ by $\Lambda({}_\phi A_\psi)= \widehat{\psi^{-1} \circ \phi}$ (recalling that $\phi\colon C(T) \to Z(A)$ is an isomorphism). We want to verify that this is a well-defined function on $\pi_1$; thus, suppose $_\phi A_\psi$ and $_{\phi'} A'_{\psi'}$ are $2$-isomorphic via $X \colon A \Rightarrow A'$. Recall that this imprimitivity bimodule $X$ satisfies, for all $f \in C(T)$ and $x\in X$,
$$\textbf{(a): } \phi(f)\triangleright x = x \triangleleft \phi'(f) \text{ and }\textbf{(b): }\psi(f)\triangleright x = x \triangleleft \psi'(f).$$
This allows us to show, for all $f \in C(T)$ and $x \in X$, that
$$\hspace{0.9cm}\psi((\psi'{^{-1}}\circ\phi')(f))\triangleright x \overset{\textbf{(b)}}{=} x \triangleleft \psi'((\psi'{^{-1}}\circ\phi')(f)) = x\triangleleft \phi'(f) \overset{\textbf{(a)}}{=} \phi(f)\triangleright x.$$
Thus, since $X$ is a full bimodule, we conclude that $(\psi\circ \psi'{^{-1}}\circ \phi')(f)=\phi(f)$ for all $f \in C(T)$. In particular, this means that $\phi^{-1}\circ \psi = \phi'{^{-1}}\circ \psi'$, so the above map is well-defined on $\pi_1$.

To show that this is a group homomorphism, given two $1$-morphisms ${}_\phi A_\psi$ and ${}_\mu B_\nu$ in $\pi_1(\AbCAlg,T)$, the 1-morphism ${}_\phi A \otimes_T B_\nu$ has associated central homomorphisms $\phi \otimes_T 1_B$ and $ 1_A \otimes_T \nu$. If we consider an element $f \in C(T)$, we have
\begin{align*}
    (1_A \otimes_T \nu)^{-1}(\phi \otimes_T 1_B)(f)&=(1_A \otimes_T \nu)^{-1}(\phi(f) \otimes_T 1_B)\\
    &=(1_A \otimes_T \nu)^{-1}(\psi(\psi^{-1}\phi(f)) \otimes_T 1_B)\\
    &=(1_A \otimes_T \nu)^{-1}(1_A \otimes_T \mu(\psi^{-1}\phi(f))\\
    &=\nu^{-1}\mu\psi^{-1}\phi(f).
\end{align*}
We therefore see that $\Lambda(A \otimes_T B)=\widehat{\nu^{-1}\mu\psi^{-1}\phi}=(\widehat{\psi^{-1}\phi})\circ (\widehat{\nu^{-1}\mu})=\Lambda(A)\circ\Lambda(b)$, and so $\Lambda$ is a group homomorphism. Furthermore, we see that this is surjective, as for any $\Phi \in {\homeo}(T)$, we have $\Lambda({_{\Phi^*}}C(T)_{\id})=\Phi$. Finally, we show that this sequence is exact at $\pi_1$. $H^3_{{\tor}}(T;\bZ)$ is identified with morphisms of the form $_{\phi}A_\phi$ by Lemma~\ref{singlecentralmap}, which is contained in the kernel of $\Lambda$. Because any morphism in $\ker(\Lambda)$ has the form $_{\phi}A_\phi$, which has a single central $\ast$-homomorphism, its equivalence class in $\pi_1(\AbCAlg, C(T))$ is determined solely by its Dixmier--Douady class. Thus, $H^3_{\tor}(T;\bZ)$ surjects onto $\ker(\Lambda)$.

\item \label{sses3} Define (iii) to be the function $\Phi \mapsto {_{\Phi^*}}C(T)_{\id}$, which is a group homomorphism by Proposition~\ref{homeoarithmetic}. As noted in defining the map (ii), we have $\Lambda({_{\Phi^*}}C(T)_{\id})=\Phi$ for any $\Phi \in {\homeo}(T)$, so we conclude (iii) is a splitting. \qedhere
\end{enumerate}
\end{proof}

\begin{rmk}
The previous short exact sequence is not trivial in general, i.e., $\pi_1(\AbCAlg,T)$ is not a direct sum $H^3_{\operatorname{tor}}(T;\bZ) \oplus \operatorname{Homeo}(T)$. Indeed, let $S\coloneq \Sigma\mathbb{RP}^2$, the suspension of the real projective plane. Note that $H^3(S;\bZ)=\bZ_2$, and so with $T\coloneq S \sqcup S$, we have $H^3(T;\bZ)=\bZ_2 \oplus \bZ_2$, which we may also write as $\{a,b | ab=ba, a^2=b^2=1\}$. Let $A$ be a unital continuous-trace C*-algebra over $S$ with $\delta(A)=a$, where the identification of spectra corresponds to $\phi\colon C(S)\to Z(A)$. Construct a $1$-automorphism of $C(T)$ as
$${_{\phi \oplus \text{id}}}A \oplus C(S)_{\phi \oplus \text{id}}$$
First, note that we have $\delta(A \oplus C(S))=a$. Now, let $\psi$ be the swap-automorphism of $C(T)=C(S)\oplus C(S)$ given by $\psi(f,g)=(g,f)$. Then
\begin{align*}
&(_{\psi} C(T)_{\id}) \otimes_T ({_{\phi \oplus \text{id}}}A \oplus C(S)_{\phi \oplus \text{id}}) \otimes_T (_{\psi} C(T)_{\id})^{-1}\\
\cong&(_{\psi} C(T)_{\id}) \otimes_T ({_{\phi \oplus \text{id}}}A \oplus C(S)_{\phi \oplus \text{id}}) \otimes_T (_{\psi^{-1}} C(T)_{\id})\\
\cong&{_{(\phi \oplus \text{id})\circ \psi}}A \oplus C(S)_{(\phi \oplus \text{id})\circ \psi} 
\end{align*}
via the isomorphism $\sigma ((f_1,f_2) \otimes_T (a, h) \otimes_T (g_1,g_2)) = (\phi(f_1)a\phi(g_2), f_2hg_1)$. This respects the central $\ast$-homomorphisms; for, given $(f,g) \in C(T)=C(S) \oplus C(S)$, we have
\begin{align*}
    (\sigma \circ(\psi \otimes_T 1_{A \oplus C(S)} \otimes_T 1_{C(T)}))(f,g)&=\sigma((g,f) \otimes_T 1_{A \oplus C(S)} \otimes_T 1_{C(T)})\\
    &=(\phi(g),f)\\
    &=((\phi \oplus \id)\circ \psi)(f,g)
\end{align*}
and
\begin{align*}
    (\sigma \circ(1_{C(T)} \otimes_T 1_{A \oplus C(S)} \otimes_T \id))(f,g)&=\sigma(1_{C(T)} \otimes_T 1_{A \oplus C(S)} \otimes_T (f,g))\\
    &=(\phi(g),f)\\
    &=((\phi \oplus \id)\circ \psi)(f,g).
\end{align*}
Thus, $\sigma$ defines a $2$-equivalence in $\AbCAlg$ by Lemma~\ref{equivariantstariso}. Now, note that for $f, g \in C(S)$, we have
$$(\phi \oplus \id)(f,g)=\phi(f) \oplus g$$
and
$$(\phi \oplus \id) \circ \psi (f,g)=\phi(g) \oplus f.$$
Therefore, the part of the C*-algebra with trivial Dixmier--Douady class -- $C(S)$ -- is now living over the first copy of $S$ in $S \sqcup S$ rather than the second copy. The part with non-trivial Dixmier--Douady class -- $A$ -- is now living over the second copy of $S$. So we have 
$$\delta({_{(\phi \oplus \text{id})\circ \psi}}A \oplus C(S)_{(\phi \oplus \text{id})\circ \psi})=b$$
which is not $a=\delta({{_{\phi \oplus \text{id}}}A \oplus C(S)_{\phi \oplus \text{id}}})$. So the two $1$-automorphisms cannot be equivalent; that is, our semidirect product is not, in general, a direct sum.

\end{rmk}

\subsection{Results for $\pi_2$ and $\pi_3$}\label{subsec:pi2}

Having completed our analysis of $\pi_1$, we move on to describe $\pi_2$ and $\pi_3$. We will begin with some important facts for line bundles. The following results are well-known in the non-unitary setting. We include these for completeness while adapting them for the unitary setting, where the same object may be equipped with potentially different unitary structures. We recall the following definition of the Picard group.
\begin{defi}
For a compact Hausdorff space $T$, the Picard group $\operatorname{Pic}(T)$ is given by (isomorphism classes of) complex line bundles over $T$ with multiplication given by the fiberwise tensor product $\otimes$.    
\end{defi}

Because we consider invertible $2$-morphisms ${_{\id}}C(T)_{\id} \Rightarrow {_{\id}}C(T)_{\id}$, our bimodules will be equipped with but a single action of $C(T)$. However, they will be imprimitivity bimodules over $C(T)$ and therefore have two $C(T)$-valued inner products $\langle \cdot | \cdot \rangle$ and $\langle \cdot, \cdot \rangle$. A priori, these inner products may be quite different, but the following lemma proves that they mutually determine each other.

\begin{lem}\label{lem:ImprimitivityBimodOnSpaceOfSections}
Let $E \to T$ be a line bundle over $T$ with two Hermitian metrics $\langle \cdot , \cdot \rangle$ and $\langle \cdot | \cdot \rangle$ which equip the space of sections $\Gamma(E)$ with the structure of a $C(T)$-$C(T)$ imprimitivity bimodule, i.e.,
$$
{}_{C(T)}\langle f , g \rangle h = f \langle g | h \rangle_{C(T)} \qquad \forall f,g,h \in \Gamma(E).
$$
Then $\langle \cdot , \cdot \rangle = \overline{\langle \cdot | \cdot \rangle}$.
\end{lem}

\begin{proof}
First, if $E = T \times \bC$ is the trivial line bundle, we may use the constant section $1 \in \Gamma(E)$ to compute, for any $f,g \in \Gamma(E)$,
$$
_{C(T)}\langle f, g \rangle = 
{_{C(T)}}\langle f, g \rangle 1 = 
f \langle g | 1 \rangle _{C(T)}
=
\langle g | f \rangle _{C(T)}
= 
\overline{\langle f | g \rangle}_{C(T)}.
$$
The case of a general line bundle $E \xrightarrow{p} T$ follows by a simple partition of unity argument. Indeed, choose a partition of unity $\{\sigma_i \colon T \to \bC\}_i$ on $T$ such that each $K_i \coloneqq \clo{\operatorname{supp}}\, \sigma_i \subseteq T$ is compact and locally trivializable, i.e. $p^{-1}(K_i) \cong K_i \times \bC$ as bundles. Applying our previous argument over each $K_i$, we deduce
\[
_{C(T)}\langle f,g \rangle = 
\sum_{i} {_{C(T)}}\langle \sigma_i f, g \rangle
=
\sum_{i} \langle \overline{\sigma_i f | g \rangle}_{C(T)}
= 
\overline{\langle f | g \rangle}_{C(T)}.
\]
for all $f,g \in \Gamma(E)$. As the inner products on $\Gamma(E)$ agree up to complex conjugation, it follows that the Hermitian metrics ${_L}\langle \cdot, \cdot \rangle$ and $\langle \cdot | \cdot \rangle_R$ agree up to complex conjugation as well.
\end{proof}

This means that our line bundles that produce imprimitivity bimodules will really only have a single choice of Hermitian metric. However, the Picard group $\operatorname{Pic}(T)$ does not involve Hermitian metrics. The next lemma shows that the choice of Hermitian metric is essentially superfluous. 

\begin{lem}\label{lem:HermitianStructuresOnBundleAreIsomorphic}
Let $E \to T$ be a line bundle with two hermitian structures $\langle \cdot | \cdot \rangle_1$ and $\langle \cdot | \cdot \rangle_2$. Then there is a unitary isomorphism between the (right) Hilbert C*-modules $(\Gamma(E),\langle \cdot | \cdot \rangle_1)$ and $(\Gamma(E),\langle \cdot | \cdot \rangle_2)$.    
\end{lem}

\begin{proof}
By Theorem 2.5 of \cite{K25}, there is an isometry $f \colon(E, \langle \cdot |\cdot\rangle_1) \to (E, \langle \cdot | \cdot\rangle_2)$. That is, $f$ is a bundle map such that $\langle f(e_1) | f(e_2)\rangle_2 = \langle e_1, e_2 \rangle_1$. Furthermore, as $f$ is also a bundle isomorphism (from \cite{K25}), we have that $f^\dagger=f^{-1}$, in the sense that
$$\langle f(e_1) | e_2 \rangle_2 = \langle e_1 | f^{-1}(e_2)\rangle _1.$$
By Lemma~\ref{lem:swandaggerfunctor}, $\Gamma$ is a $\dagger$-functor, so $\Gamma(f)\colon (\Gamma(E),\langle \cdot | \cdot \rangle_1) \to (\Gamma(E),\langle \cdot | \cdot \rangle_2)$ is a unitary isomorphism.
\end{proof}

Having justified that the pair of Hermitian metrics on $E$ that yield an imprimitivity bimodule $\Gamma(E)$ do not affect the isomorphism class of $\Gamma(E)$, the characterization of $\pi_2$ as $\operatorname{Pic}(T)$ follows from the monoidal version of the Serre--Swan Theorem. We produce the precise statement of this equivalence in the below corollary, but we refer the reader to Appendix~\ref{Sec:Serre-Swan} for the proof.

\begin{cor}[Serre--Swan]\label{cor:Serre-Swan}
There is a monoidal $\dagger$-equivalence $E \mapsto \Gamma(E)$ between the monoidal $\dagger$-categories of Hermitian line bundles over $T$ and Hilbert $C(T)$-bimodules.
\end{cor}

We can now prove that $\pi_2(\AbCAlg,C(T))$ is isomorphic to $\operatorname{Pic}(T)$.
\begin{thm}
For $T \in \AbCAlg$, $\pi_2(\AbCAlg,T) \cong \operatorname{Pic}(T)$. Furthermore, when $T$ has the homotopy type of a CW-complex, $\pi_2(\AbCAlg, T) \cong H^2(T;\bZ)$.    
\end{thm}

\begin{proof}
By the proposition on page 291 of \cite{R80}, a $C(T)$-$C(T)$ imprimitivity bimodule is finitely generated and projective. Thus, using Lemma \ref{lem:ImprimitivityBimodOnSpaceOfSections} and Corollary \ref{cor:Serre-Swan}, we see that unitary isomorphism classes of imprimitivity bimodules over $C(T)$ are in correspondence with unitary isomorphism classes of Hermitian line bundles over $T$. As every complex line bundle over the compact space $T$ admits a Hermitian metric (\cite[Proposition~1.2]{H03}), Lemma \ref{lem:HermitianStructuresOnBundleAreIsomorphic} proves that the latter are in correspondence with isomorphism classes of line bundles over $T$, i.e., $\pi_2(\AbCAlg,T) \cong \operatorname{Pic}(T)$. Since the correspondence from Corollary \ref{cor:Serre-Swan} is monoidal, these are indeed isomorphic as groups.
\par Finally, when $T$ has the homotopy type of a CW-complex, by Proposition 3.10 of \cite{H03}, the map sending a line bundle $E$ to its first Chern class $c_1(E)$ is an isomorphism from $\text{Pic}(T)$ to $H^2(T;\bZ)$.
\end{proof}

Proving that $\pi_3(\AbCAlg, C(T))=C(T)^\times$ is straightforward, especially because we have ``topped-out" the structure of $\AbCAlg$ and the only equivalence relation on our $3$-morphisms is equality.

\begin{thm}
For $T \in \AbCAlg$, $\pi_3(\AbCAlg,T) \cong C(T)^\times$.
\end{thm}

\begin{proof}
Recall that $\id_{\id_T} = C(T)$ as a Hilbert C*-bimodule over itself. Since $C(T)$ is abelian and unital, 
$$\End({}_{C(T)}C(T)_{C(T)}) \cong Z(C(T)) = C(T).$$ Hence, the invertible maps are given by $\pi_3(\AbCAlg,C(T)) \cong C(T)^\times$.
\end{proof}

\section{Actions on \texorpdfstring{$\pi_2$}{2nd} and \texorpdfstring{$\pi_3$}{3rd homotopy groups}}\label{section:actions}

Having computed the three homotopy groups $\pi_1,\pi_2$ and $\pi_3$, we are now ready to compute some of the additional data that classifies this homotopy $3$-type. We give the following definition of the actions of $\pi_1(\AbCAlg, C(T))$ on $\pi_2(\AbCAlg, C(T))$ and $\pi_3(\AbCAlg, C(T))$.

\begin{defi}
    We define the actions of $\pi_1$ on $\pi_2$ and $\pi_3$ as follows: Let ${_\phi}A_\psi$ be in $\pi_1(\AbCAlg,C(T))$, and choose a $2$-isomorphism $Y\colon A \otimes_T {_{\id}}C(T)_{\id}\otimes_T A^{-1} \Rightarrow{_{\id}}C(T)_{\id}$. For $X \in \pi_2(\AbCAlg,C(T))$, we define $A \curvearrowright X$ as the bimodule
    $$Y^{-1}\otimes (\id_A \otimes_T X \otimes_T \id_{A^{-1}}) \otimes Y.$$
    This action is well-defined as both $\pi_1$ and $\pi_2$ are defined up to equivalence. For $f \in \pi_3(\AbCAlg, C(T))$, we define $A \curvearrowright f$ using the $3$-morphism
    $$\tilde{f}\coloneq\id_{Y^{-1}} \otimes (1_A \otimes_T f \otimes_T 1_{A^{-1}}) \otimes \id_{Y}$$
    as a $3$-automorphism of $A \curvearrowright \id_{C(T)}$. However, in general, $A \curvearrowright \id_{C(T)}$ is not equal to $\id_{C(T)}$, only isomorphic. Thus, we choose a $3$-isomorphism $g\colon A \curvearrowright\id_{C(T)} \Rrightarrow \id_{C(T)}$ and define $A\curvearrowright f \coloneq g \circ \tilde{f} \circ g^{-1}$. Again, this is well-defined, as everything is up to equivalence.
\end{defi}

\begin{rmk}\label{isomorphismmoduleactions}
    Recall that $A \otimes_T {_{\id}}C(T)_{\id} \otimes_T A^{-1}$ is $\ast$-isomorphic to ${_{\id}}C(T)_{\id}$ in a way that respects the action, so we can construct an invertible bimodule $Y$ from this isomorphism by Lemma~\ref{equivariantstariso}. In this case, $2$-composing with $Y$ and $Y^{-1}$ simply yields $\id_A \otimes_T X \otimes_T \id_{A^{-1}}$ as a $C(T)$-$C(T)$ bimodule with the actions given by the $\ast$-isomorphism.
\end{rmk}

\subsection{Actions by Cohomology}

Because $H^3_{\tor}(T;\bZ)$ is an abelian group, intuition suggests that we should be able to braid the $1$-morphisms in the definition of the actions. This is made precise in the following lemma.

\begin{lem}
    Let $A\coloneq {_\phi}A_\phi\in \pi_1(\AbCAlg, C(T))$. Then the functor defined by $X \mapsto Y^{-1} \otimes(\id_A \otimes_T X \otimes_T \id_{A^{\op}}) \otimes Y$ and $f \mapsto \id_{Y^{-1}} \otimes(1_A \otimes_T f \otimes_T 1_{A^{\op}})\otimes \id_Y$ is naturally unitarily isomorphic to the identity functor $($for $X \in \pi_2(\AbCAlg, C(T))$ and $f \in \pi_3(\AbCAlg, C(T)))$.
\end{lem}

\begin{proof}
    We have that $A \otimes_T {_{\id}}C(T)_{\id} \otimes_T A^\op$ is $\ast$-isomorphic to $A \otimes_T A^\op$ via the map
    $$a \otimes_T f \otimes_T b \mapsto a\phi(f)\otimes_T b.$$
    This $\ast$-isomorphism preserves the central $\ast$-homomorphisms, and thus induces a $2$-equivalence $$Z\colon A \otimes_T {_{\id}}C(T)_{\id} \otimes_T A^\op \Rightarrow A \otimes_T A^\op$$ by Lemma~\ref{equivariantstariso}. By Remark~\ref{isomorphismmoduleactions}, we then have that $Z^{-1} \otimes (\id_A \otimes_T X \otimes_T \id_{A^{\op}})\otimes Z$ is isomorphic to $\id_A \otimes_T X \otimes_T \id_{A^\op}$ as an $A\otimes_T A^{\op}$ bimodule via the actions
    $$(a \otimes_T b) \triangleright (c \otimes_T x \otimes_T d)= ac \otimes_T x \otimes_T db$$
    and
    $$ (c \otimes_T x \otimes_T d) \triangleleft (a \otimes_T b) = ca \otimes_T x \otimes_T bd,$$
    and is furthermore equipped with the inner products
    $$\langle a \otimes_T x \otimes_T b | c \otimes_T y \otimes_T d\rangle = a^\dagger c \phi(\langle x | y \rangle) \otimes_T db^\dagger$$ 
    and
    $$\langle a \otimes_T x \otimes_T b , c \otimes_T y \otimes_T d\rangle = ac^\dagger \phi(\langle x , y \rangle )\otimes_T d^\dagger b.$$ 
    Note that multiplication in the third tensor factor is reversed because the elements lie in $A^\op$. We first claim that $Z^{-1} \otimes (\id_A \otimes_T X \otimes_T \id_{A^\op}) \otimes Z$ is naturally $3$-isomorphic to the $A\otimes_T A^\op$ bimodule $X\otimes_T \id_A \otimes_T \id_{A^\op}$, where the actions and inner products are given similarly. The only subtlety here is that the tensor permutation map $\tau$ given by $\tau(a \otimes_T x \otimes_T b)= x \otimes_T a \otimes_T b$ is well-defined. Well-definedness of this map heavily relies upon the fact that ${_\phi}A_\phi$ has identical central homomorphisms. For, given $f \in C(T)$, we have 

\[
\begin{tikzcd}[row sep=7pt]
a \otimes_T x \otimes_T \phi(f)b \arrow[r, mapsto, "\tau"] \arrow[d,equal]& x \otimes_T a \otimes_T \phi(f)b \arrow[d,equal] \\
a \phi(f) \otimes_T x \otimes_T b \arrow[r,mapsto,"\tau"] \arrow[d,equal]& x \otimes_T a \phi(f) \otimes_T b \arrow[d,equal]\\
a \otimes_T (f \rhd x) \otimes_T b \arrow[r,mapsto,"\tau"] & (f \rhd x) \otimes_T a \otimes_T b
\end{tikzcd}
\] 
    Furthermore, because the inner products of simple tensors in $\id_A \otimes_T X \otimes_T \id_{A^\op}$ and their images under $\tau$ produce the same elements of $A \otimes_T A^\op$, we see that $\tau$ is isometric on the span of simple tensors and therefore extends to all of $\id_A \otimes_T X \otimes_T \id_{A^\op}$ by continuity. Furthermore, we see that $\tau$ is adjointable with $\tau^\dagger=\tau^{-1}$, for
    \begin{align*}
        \langle\tau(a \otimes_T x \otimes_T b) | y \otimes_T c \otimes_T d\rangle&=\langle x \otimes_T a \otimes_T b | y \otimes_T c \otimes_T d\rangle\\
        &=a^\dagger c \langle x | y \rangle \otimes_T db^\dagger\\
        &=\langle a \otimes_T x \otimes_T b | c \otimes_T y \otimes_T d\rangle\\
        &=\langle a \otimes_T x \otimes_T b | \tau^{-1}(y \otimes_T c \otimes_T d)\rangle
    \end{align*}
It is routine to see that $\tau$ is natural; the following diagram clearly commutes for any $3$-morphism $g \colon X \to X'$.
\begin{center}
    \begin{tikzcd}
        \id_A \otimes_T X \otimes_T \id_{A^\op} \arrow[d, "\tau"] \arrow[rrr, "1_A \otimes_T g \otimes_T 1_{A^\op}"] &&& \id_A \otimes_T X' \otimes_T \id_{A^\op} \arrow[d, "\tau'"]\\
        X \otimes_T \id_A \otimes_T \id_{A^\op} \arrow[rrr, "g \otimes_T 1_A \otimes_T 1_{A^\op}"'] &&& X' \otimes_T \id_A \otimes_T \id_{A^\op}
    \end{tikzcd}
\end{center}
Now, choose a $2$-equivalence $Y'\colon A \otimes_T A^\op \Rightarrow {_{\id}}C(T)_{\id}$, and set $Y\coloneq Z \otimes Y'$. We then have natural isomorphisms between the following:
\begin{align*}
Y^{-1} \otimes (\id_A \otimes_T X \otimes_T \id_{A^{\op}}) \otimes Y &\cong (Y')^{-1} \otimes Z^{-1} \otimes (\id_A \otimes_T X \otimes_T \id_{A^{\op}}) \otimes Z \otimes Y'\\
&\cong ({Y'})^{-1} \otimes  (X \otimes_T \id_A \otimes_T \id_{A^{\op}})  \otimes Y'\\
&\cong
X \otimes_{T} ((Y')^{-1} \otimes Y')\\ 
&\cong 
X \otimes_{T} C(T)\\
&\cong 
X
\end{align*}
These isomorphisms are all natural because $1$-composition is natural and the only maps used to construct the isomorphisms (other than $\tau$) are unitors, associators, and interchangers, which are all natural as well. Furthermore, $\tau$ is a unitary isomorphism, and all coherence data in $\AbCAlg$ is unitary, so this natural isomorphism is, in fact, unitary.
\end{proof}
Because these $1$-morphisms may be braided, the actions given by morphisms in $H^3_{\tor}(T;\bZ)$ are automatically trivial.
\begin{cor}
    Let ${_\phi}A_\phi$ be in $\pi_1(\AbCAlg, C(T))$. Then the actions of $_\phi A_\phi$ on $\pi_2(\AbCAlg, C(T))$ and $\pi_3(\AbCAlg, C(T))$ are trivial.
\end{cor}

\subsection{Actions by Homeomorphisms}
In this subsection we will describe the actions by $1$-morphisms in ${\homeo}(T)$. Because the central $\ast$-isomorphisms differ, these morphisms cannot be braided as we did the morphisms in $H^3_{\tor}(T;\bZ)$. The following lemma describes the image of an invertible bimodule $X$ under this action.
\begin{lem}\label{homeopi2}
    For a $1$-morphism $A\coloneq {_{\phi}}C(T)_{\id}\in \pi_1(\AbCAlg,C(T))$ and $X \in \pi_2(\AbCAlg, C(T))$, we claim that $\id_A \otimes_T X \otimes_T \id_{A^{-1}}$ is isomorphic to the bimodule $\phi(X)$, which has the same underlying vector space $X$ with action $\blacktriangleright$ and inner product $(\cdot | \cdot)$ defined as follows:
    $$f \blacktriangleright x = \phi(f) \triangleright x$$
    $$(x| y) = \phi^{-1}(\langle x | y \rangle)$$
    where $\triangleright $ and $\langle \cdot| \cdot \rangle$ are the original actions and inner product on $X$.
\end{lem}

\begin{proof}
    First, recall that $A^{-1}={_{\phi^{-1}}}C(T)_{\id}$. Furthermore, remember that $\id_A$ is $C(T)$ as a $C(T)$-$C(T)$ with the usual inner products and actions. Now, observe that
    $$A\otimes_T C(T) \otimes_T A ^{-1}\cong C(T)$$
    via the isomorphism 
    $$\sigma\colon f \otimes g \otimes h \mapsto \phi^{-1}(fg)h .$$
    Furthermore, this isomorphism respects the central $\ast$-homomorphisms from $C(T)$. For, given $f \in C(T)$,
    $$\sigma (\phi \otimes_T 1_{C(T)} \otimes_T 1_{A^{-1}})(f)= \sigma(\phi(f) \otimes_T 1_{C(T)} \otimes_T 1_{A^{-1}}) = f$$
    and
    $$\sigma (1_A \otimes_T 1_{C(T)} \otimes_T {\id})(f)=\sigma (1_A \otimes_T 1_{C(T)} \otimes_T f)=f,$$
    and so $\sigma$ is a $\ast$-isomorphism from $A \otimes_T C(T) \otimes_T A^{-1}$ to ${_{\id}}C(T)_{\id}$ that respects the central maps. Thus, the corresponding imprimitivity bimodule $Y$ will be a $2$-isomorphism in $\AbCAlg$ by Lemma~\ref{equivariantstariso}. By Remark~\ref{isomorphismmoduleactions}, we have
    \begin{align*}
        A \curvearrowright X &= Y^{-1}\otimes (\id_A \otimes_T X \otimes_T \id_{A^{-1}}) \otimes Y\\
        &\cong \id_A \otimes_T X \otimes_T \id_{A^{-1}}
    \end{align*}
    where the resulting bimodule $\id_A \otimes_T X \otimes_T \id_{A^{-1}}$ as a $C(T)$-$C(T)$ bimodule with actions defined using $\sigma$. Furthermore, it is clear that there is a vector space isomorphism $\Sigma \colon \id_{A} \otimes_T X \otimes_T \id_{A^{-1}} \to X$ given by $\Sigma(f \otimes_T x \otimes_T g)=f\phi(g)\triangleright x$. Thus, we define $\phi(X)$ to have underlying vector space $X$ given by this isomorphism. We use the maps $\sigma$ and $\Sigma$ to determine the $C(T)$-action on $\phi(X)$ as follows:
    \begin{align*}
        f \blacktriangleright x& = \Sigma( \sigma^{-1}(f) \triangleright \Sigma^{-1}(x))\\
        &= \Sigma((1_A \otimes_T \phi(f) \otimes_T 1_{A^{-1}}) \triangleright (1_A \otimes_T x \otimes_T 1_{A^{-1}}))\\
        &= \Sigma(1_A \otimes_T (\phi(f)\triangleright x) \otimes_T 1_{A^{-1}})\\
        &= \phi(f)\triangleright x
    \end{align*}
    We define the inner product in a similar manner:
    \begin{align*}
        ( x | y) &= \sigma(\langle \Sigma^{-1}(x)| \Sigma^{-1}(y)\rangle)\\
        &= \sigma(\langle1_A \otimes_T x \otimes_T 1_{A^{-1}}| 1_A \otimes_T y \otimes_T 1_{A^{-1}}\rangle)\\
        &= \sigma(\langle 1_A | 1_A \rangle \otimes_T \langle x | y \rangle \otimes_T \langle 1_{A^{-1}}| 1_{A^{-1}}\rangle )\\
        &= \phi^{-1}(\langle x | y \rangle)
    \end{align*}
    With these definitions, $\Sigma \colon (A \curvearrowright X) \to \phi(X)$ is manifestly a $3$-isomorphism in $\AbCAlg$.
\end{proof}

Having computed how $1$-morphisms in $H^3_{\tor}(T;\bZ)$ and ${\homeo}(T)$ individually act on $\pi_2(\AbCAlg, C(T))$, we can now describe how a general $1$-morphism acts.

\begin{cor}\label{1moractiononpi2}
    Let $_\phi A_\psi$ belong to $\pi_1(\AbCAlg, C(T))$ and $X$ to $\pi_2(\AbCAlg, C(T))$. Then $A \curvearrowright X$ is the bimodule $\psi^{-1}\phi(X)$ given by the same underlying vector space $X$ with actions and inner product as follows:
    $$f \blacktriangleright x = (\psi^{-1}\circ \phi)\triangleright x$$
    $$(x| y ) = (\phi^{-1} \circ \psi)(\langle x | y \rangle).$$
\end{cor}

\begin{proof}
    By Proposition~\ref{prop:splitoffcenter}, $_\phi A_\psi$ is equivalent to $_{\psi^{-1}\phi} C(T)_{\id} \otimes_T {_\psi}A_\psi$ in the first homotopy group $\pi_1(\AbCAlg, C(T))$. By our previous work, we have $$_\phi A_\psi \curvearrowright X \cong {_{\psi^{-1}\phi}}C(T)_{\id} \curvearrowright (_{\psi}A_\psi \curvearrowright X) \cong {_{\psi^{-1}\phi}}C(T)_{\id} \curvearrowright X\cong \psi^{-1}\phi(X)$$
    as desired.
\end{proof}

We would also like to describe the actions on $\pi_2$ in terms of line bundles and, in particular, in terms of the first Chern class in $H^2(T;\bZ)$. The following theorem describes this relationship.
\begin{thm}
    Let $(\omega, \Phi) \in H^3_{\tor}(T;\bZ) \rtimes {\homeo}(T)$, and let $E$ be a line bundle over $T$. Then $(\omega, \Phi) \curvearrowright E$ is the pullback bundle $(\Phi^{-1})^*(E)$. Furthermore, when $T$ has the homotopy type of a CW-complex, the action of $(\omega, \Phi)$ on $H^2(T;\bZ)$ corresponds to the pullback along $\Phi^{-1}$.
\end{thm}

\begin{proof}
    By the previous theorem, we know that $(\omega, \Phi) \curvearrowright \Gamma(E) \cong \Phi^*(\Gamma(E))$, where the $C(T)$-action has been twisted by $\Phi^*$ (that is, by precomposition with $\Phi$). Define the pullback bundle $(\Phi^{-1})^*(E)$ to have total space
    $$(\Phi^{-1})^*(E)\coloneq \{(t, e) \in T \times E: \Phi^{-1}(t)=p(e)\} \subseteq T \times E$$
    with projection map $p_{\Phi}(t,e) = t = (\Phi \circ p)(e).$ We see that precomposition with $\Phi^{-1}$ defines a group isomorphism from $\Phi^*(\Gamma(E))$ to $\Gamma((\Phi^{-1})^*(E))$ by sending a continuous section $g \in \Gamma(E)$ to $(g \circ \Phi^{-1})(t) = (t, (g \circ \Phi^{-1})(t))$. To verify this, note this does produce a section in $\Gamma((\Phi^{-1})^*E)$; for, if $g \in \Gamma(E)$, we have
    $$(\Phi \circ p) \circ (g \circ \Phi^{-1})=\Phi \circ \id_T \circ \Phi^{-1} = \id_T.$$
    Also, this map has an inverse sending $h \in \Gamma((\Phi^{-1})^*(E))$ to $h \circ \Phi$ (with a mild abuse of notation by identifying $E$ with $\id^*(E)$). Note that $(\Phi^{-1})^*\colon \Phi^*(\Gamma(E))\to \Gamma(\Phi^{-1})^*(E))$ also intertwines the $C(T)$-module actions; for any $f \in C(T)$, we have
    $$({\Phi^{-1}})^*(f \blacktriangleright g) = (\Phi^{-1})^*((f \circ \Phi)g)= f (g \circ \Phi^{-1})= f \triangleright (\Phi^{-1})^*(g).$$
    We conclude that $(\Phi^{-1})^*$ is an isomorphism of $C(T)$-modules, and it follows that $(\omega, \Phi)\curvearrowright E \cong (\Phi^{-1})^*(E)$.
    \par Furthermore, in the case where $T$ has the homotopy type of a CW-complex, we know that line bundles are classified by their first Chern class $c_1(E) \in H^2(T;\bZ)$. We then have that $(\omega, \Phi) \curvearrowright c_1(E)= (\Phi^{-1})^* c_1(E)$, the pullback of the first Chern class (as the pullback of line bundles corresponds to the pullback of the Chern classes). 
\end{proof}

We are now ready to analyze the action of $1$-morphisms in ${\homeo}(T)$ on the third homotopy group $\pi_3(\AbCAlg, C(T))$. As we can concretely describe the image of the identity bimodule $\id_{C(T)}$ under this action, we essentially trace how a morphism in $\pi_3$ acts through these isomorphisms. 

\begin{thm}
    Let $f \in \pi_3(\AbCAlg, T)$ and $A\coloneq {_\phi} C(T)_{\id} \in \pi_1(\AbCAlg, T)$. Then $A\curvearrowright f = \phi^{-1}(f)$.
\end{thm}

\begin{proof}
   By definition, we have $A \curvearrowright f = \id_{Y^{-1}} \otimes (1_A \otimes_T f \otimes_T 1_{A^{-1}}) \otimes \id_{Y} $. However, because we can choose $Y$ to be implemented by an isomorphism $\sigma$ as in the proof of Lemma~\ref{homeopi2}, we can determine how $A \curvearrowright f$ acts on our representative $\phi(\id_{C(T)})$ by composing $\Sigma \circ (1_A \otimes_T f \otimes_T 1_{A^{-1}}) \circ \Sigma^{-1}$ (where $\Sigma$ also comes from the proof of Lemma~\ref{homeopi2}). We then see that
   \begin{align*}
       (\Sigma \circ (1_A \otimes_T f \otimes_T 1_{A^{-1}}) \circ \Sigma^{-1})(g) &= \Sigma(1_A \otimes_T f \otimes_T 1_{A^{-1}})(1_A \otimes_T g \otimes_T 1_{A^{-1}})\\
       &= \Sigma(1_A \otimes_T fg \otimes_T 1_{A^{-1}})\\
       &= fg.
   \end{align*}
   Thus, $A \curvearrowright f$ acts on $\phi(\id_{C(T)})$ by multiplying by $f$. However, $\phi(\id_{C(T)})$ is not explicitly equal to $\id_{C(T)}$, which is the $2$-morphism on which the elements of $\pi_3$ act. Therefore, we need to use an appropriate isomorphism $\phi(\id_{C(T)}) \Rrightarrow \id_{C(T)}$. Now, consider the map $\eta \colon \phi(\id_{C(T)}) \to \id_{C(T)}$ given by
   $$\eta(g)=\phi^{-1}(g).$$
   This map is clearly invertible with $\eta^\dagger=\eta^{-1}$, for
   $$\langle \eta(g) | h \rangle = \phi^{-1}(\clo{g})h=\phi^{-1}(\clo{g} \phi(h))=(g | \phi(h))=(g | \eta^{-1}(h)).$$
   So this is a $3$-equivalence from $\phi(\id_{C(T)})$ to $\id_{C(T)}$. To see how $A \curvearrowright f$ acts on $\id_{C(T)}$, we compute (for $g \in \id_{C(T)}$):
   \begin{align*}
        (\eta \circ (A \curvearrowright f) \circ \eta^{-1})(g)&= (\eta \circ (A \curvearrowright f))(\phi(g))\\
        &= \eta ( f\phi(g))\\
        &= \phi^{-1}(f)g.
   \end{align*}
   Thus, we see that $A \curvearrowright f$ is equal to $\phi^{-1}(f)$ in $\pi_3$.
\end{proof}

As was the case with $\pi_2$, we can also describe the action of a general invertible $1$-morphism on $\pi_3$.

\begin{cor}\label{generalpi3action}
    Given ${_\phi}A_\psi \in \pi_1(\AbCAlg,C(T))$ and $f \in \pi_3(\AbCAlg,C(T))$, we have that ${_\phi}A_\psi \curvearrowright f= (\psi^{-1}\circ \phi)(f)$.
\end{cor}

\begin{proof}
    As in the proof of Corollary~\ref{1moractiononpi2}, we use Proposition~\ref{prop:splitoffcenter} to write $_\phi A _\psi \cong {_{\psi^{-1}\circ \phi}}C(T)_{\id}\otimes_T {_\psi}A_\psi$. Therefore,
    \begin{align*}
        {_\phi}A_\psi \curvearrowright f &= ({_{\psi^{-1}\circ \phi}}C(T)_{\id} \otimes_T {_\psi}A_\psi) \curvearrowright f\\
        &= {_{\psi^{-1}\circ \phi}}C(T)_{\id} \curvearrowright f\\
        &= (\phi^{-1}\circ \psi)(f)\qedhere
    \end{align*}
\end{proof}
To give a characterization in terms of topological data, the next corollary immediately follows from Corollary~\ref{generalpi3action}.
\begin{cor}
    Given $(\omega, \Phi) \in H^3_{\tor}(T;\bZ)\rtimes {\homeo}(T)$ and $f \in \pi_3(\AbCAlg, T)$, we have $(\omega, \Phi) \curvearrowright f = f \circ \Phi^{-1}$.
\end{cor}

\appendix

\section{Serre--Swan duality}\label{Sec:Serre-Swan}
Swan's theorem \cite[p.267]{S62} states the construction $\Gamma$ taking a vector bundle $E$ to its space of sections $\Gamma(E)$ is an equivalence of suitable categories. Each category carries a monoidal product -- the tensor product of bundles and the relative tensor product of modules, respectively -- and so we will give a proof that the equivalence guaranteed by Swan's theorem is, in fact, a monoidal $\dagger$-equivalence. This result is already known to experts; section 7.5 of \cite{C08} gives a proof of this monoidal equivalence for the smooth version of the Serre--Swan Theorem for differentiable manifolds. Our proof for Swan's theorem is essentially a modified proof of the one found in \cite{C08}.

Returning to vector bundles, notice the rank $n$ trivial bundle $T \times \bC^n \to T$ yields the free $n$-dimensional $C(T)$-module 
$$
\Gamma(T \times \bC^n) = C(T) \otimes_\bC \bC^n.
$$
Of course, every free $C(T)$-module arises in this way, and each map of free modules $C(T) \otimes_\bC \bC^n \to C(T) \otimes_\bC \bC^m$ is uniquely determined by a map $\bC^n \to \bC^m$ which in turn induces a corresponding bundle map $T \times \bC^n \to T \times \bC^m$. In general, every finite rank bundle can be witnessed inside one of these trivial bundles:

\begin{lem}\label{lem:trivialsummand}
Let $E$ be a bundle over $T$. Then there is a bundle $E^\perp$ with the property that $E \oplus E^\perp$ is a trivial bundle over $T$.
\end{lem}

As a direct consequence, $\Gamma(E)$ is a finitely generated projective $C(T)$-module. The content of Swan's theorem is that every finitely generated projective module arises in this way, so that the categories of finite rank vector bundles over $T$ and finitely generated projective modules over $C(T)$ are equivalent.

However, in this note we are interested in Hilbert C*-modules over $C(T)$. It is well-known that the data of a $C(T)$-valued inner product on the space of sections $\Gamma(E)$ is precisely the same data as a Hermitian metric on $E$.
More specifically, $\Gamma \colon \Hilb_{\mathsf{fd}}(T) \to \Hilb_{\mathsf{fgp}}(C(T))$ is an equivalence of categories where $\Hilb_{\mathsf{fd}}(T)$ is the category of finite-dimensional vector bundles over $T$ and $\Hilb_{\mathsf{fgp}}(C(T))$ is the category of finitely generated projective $C(T)$-modules. We remark here that $\Gamma$ is a $\dagger$-functor, described precisely in the following lemma. The proof is immediate from the definition of the $C(T)$-valued inner product on $\Gamma(E)$.

\begin{lem}\label{lem:swandaggerfunctor}
    Let $E$ and $F$ be vector bundles with Hermitian metrics $\langle \cdot | \cdot \rangle_E$ and $\langle \cdot | \cdot \rangle_F$. If $\sigma\colon E \to F$ is an adjointable map of bundles in the sense that there is a bundle map $\sigma^\dagger\colon F \to E$ such that $\langle \sigma(e) | f\rangle_F = \langle e | \sigma^\dagger(f)\rangle_E $, then $\Gamma(\sigma)^\dagger = \Gamma(\sigma^\dagger)$.
\end{lem}

However, both $\Hilb_{\mathsf{fd}}(T)$ and $\Hilb_{\mathsf{fgp}}(C(T))$ admit monoidal structures $\otimes$ and $\otimes_{C(T)}$ respectively. It remains to show this equivalence is monoidal; that is, we will exhibit a unitary
$$\eta\colon C(T)_{C(T)}\to \Gamma(T \times \bC)$$
and, for each pair of vector bundles $E$ and $F$, a unitary
$$\mu_{E,F}\colon \Gamma(E) \otimes_{C(T)} \Gamma(F) \to \Gamma(E \otimes F)$$ 
subject to the following coherences:
\begin{itemize}
    \item For any bundles $E, F$, and $G$, the following associativity diagram commutes:
\begin{equation}\label{eq:assoc}
       \begin{tikzcd}
            (\Gamma(E) \otimes_{C(T)} \Gamma(F)) \otimes_{C(T)} \Gamma(G) \arrow[r, "\alpha"] \arrow[d, "\mu_{E,F} \otimes \id"'] & \Gamma(E) \otimes_{C(T)} (\Gamma(F) \otimes_{C(T)} \Gamma(G)) \arrow[d, "\id \otimes \mu_{F,G}"]\\
           \Gamma(E \otimes F) \otimes_{C(T)} \Gamma(G) \arrow[d, "\mu_{E \otimes F, G}"'] & \Gamma(E) \otimes_{C(T)} \Gamma(F \otimes G) \arrow[d, "\mu_{E, F \otimes G}"] \\
           \Gamma((E \otimes F) \otimes G) \arrow[r, "\Gamma(\alpha)"'] & \Gamma(E \otimes (F \otimes G))
        \end{tikzcd}
\end{equation}
    \item For any bundle $E$, the following two unitality diagrams commute (corresponding to left and right unitors, respectively):
\begin{equation}\label{eq:leftunital}
        \begin{tikzcd}
            C(T) \otimes_{C(T)} \Gamma(E) \arrow[r, "\eta \otimes \id"] \arrow[d, "\lambda_{\Gamma(E)}"'] & \Gamma((T \times \bC)) \otimes_{C(T)} \Gamma(E) \arrow[d, "\mu_{(T \times \bC), E}"] \\
            \Gamma(E) & \Gamma((T \times \bC) \otimes E) \arrow[l, "\Gamma(\lambda_{E})"]
        \end{tikzcd}
\end{equation}
\begin{equation}\label{eq:rightunital}
          \begin{tikzcd}
            \Gamma(E) \otimes_{C(T)} C(T) \arrow[r, "\id \otimes \eta "] \arrow[d, "\rho_{\Gamma(E)}"'] &   \Gamma(E) \otimes_{C(T)} \Gamma((T \times \bC))\arrow[d, "\mu_{E, (T \times \bC)}"] \\
            \Gamma(E) & \Gamma(E \otimes (T \times \bC)) \arrow[l, "\Gamma(\rho_{E})"]
        \end{tikzcd}
\end{equation}
\end{itemize}

\begin{defi} 
        Let $E$ and $F$ be vector bundles over the compact Hausdorff space $T$. Define the linear function $\mu_{E,F}\colon \Gamma(E)\otimes_{C(T)}\Gamma(F) \to \Gamma(E \otimes F)$ by (for $f \in \Gamma(E), g \in \Gamma(F)$)
        $$\mu_{E,F}(f \otimes_{C(T)} g)=f \otimes g.$$
        Also, define $\eta\colon C(T) \to  \Gamma(T \times \bC)$ to be, for $h \in C(T)$,
        $$\eta(h)(t)=(t, h(t)).$$
\end{defi}
It is almost a tautology that $\eta$ is a unitary isomorphism of $C(T)$-modules. Seeing that $\mu_{E,F}$ is unitary is much more subtle. We will break this proof into multiple parts.
    \begin{prop} \label{lem:trivialunitaryiso}
        If $E$ and $F$ are trivial bundles equipped with Hermitian metrics $\langle \cdot | \cdot \rangle_E$ and $\langle \cdot | \cdot \rangle_F$, then $\mu_{E,F}$ is a unitary isomorphism. 
    \end{prop}

    \begin{proof}
        Choose (orthogonal) global sections $\{f_1,f_2,\dots,f_n\}$ of $E$ and $\{g_1,g_2,\dots,g_m\}$ of $F$ that trivialize their respective bundles. These form bases for $\Gamma(E)$ and $\Gamma(F)$, respectively. Now, note that the set 
        $$\{f_i\otimes_{C(T)} g_j\}_{(i,j)=(1,1)}^{(n,m)}$$
        has dense $C(T)$-span in $\Gamma(E)\otimes_{{C(T)}}\Gamma(F)$. However, observe that
        $$\{f_i\otimes g_j\}_{(i,j)=(1,1)}^{(n,m)}$$
        is a set of (orthogonal) global sections of $E \otimes F$ that trivializes the bundle. Since $\mu_{E,F}$ carries the first set to the second, we see that $\mu_{E,F}$ must be surjective. It is clear that $\mu_{E,F}$ has a set-theoretic inverse. To show that $\mu_{E,F}^\dagger = \mu_{E,F}^{-1}$, we have, for all $1\leq i,k \leq n$ and $1\leq j, l\leq m$,
        \begin{align*}
            \langle \mu_{E,F}(f_i \otimes_T g_j) | f_k \otimes g_l \rangle (t) &= \langle (f_i \otimes g_j)(t) | (f_k \otimes g_l)(t) \rangle_{E \otimes F}\\
            &=\langle f_i(t) | f_k(t) \rangle_E\langle g_j(t) | g_l(t) \rangle_F \\
            &=\langle g_j(t) | \langle f_i(t) | f_k(t) \rangle_Eg_l(t) \rangle_F \\
            &=\langle g_j(t) | (\langle f_i | f_k \rangle\triangleright g_l)(t) \rangle_F \\
            &=\langle g_j | \langle f_i | f_k \rangle\triangleright g_l \rangle(t) \\         
            &=\langle f_i \otimes_Tg_j | \mu^{-1}_{E,F}(f_k\otimes g_l)\rangle (t)
        \end{align*}
        This shows that $\mu_{E,F}$ is adjointable and, in particular, is unitary.
    \end{proof}

    \begin{lem}\label{lem:daggerinclusion}
        Let $E$ and $F$ be vector bundles over $T$. Define the maps
        $$i \colon E \to E \oplus F \hspace{1cm} i(v)=(v,0)$$
        and
        $$\rho \colon E \oplus F \to E \hspace{1cm} \rho (v,w)=v.$$
        Then $\Gamma(\rho) \circ \Gamma(i)=\id_{\Gamma(E)}$. Furthermore, if $E$ and $F$ have Hermitian metrics $\langle \cdot | \cdot \rangle_E$ and $\langle \cdot | \cdot \rangle_F$, then $\Gamma(\rho)$ and $\Gamma(i)$ are adjointable with $\Gamma(\rho)^\dagger=\Gamma(i)$ when $E \oplus F$ is equipped with the Hermitian metric $\langle \cdot | \cdot \rangle_{E \oplus F}=\langle \cdot | \cdot \rangle_E + \langle \cdot | \cdot \rangle_F$.
    \end{lem}

    \begin{proof}
        It is clear that $\Gamma(\rho) \circ \Gamma(i)=\id_{\Gamma(E)}$, because $\rho \circ i=\id$, and $\Gamma$ is functorial. It is also routine to see that $\rho^\dagger=i$, as
        $$\langle \rho(f_1,g ) | f_2\rangle_E = \langle f_1 | f_2 \rangle_E = \langle (f_1,g) | (f_2,0) \rangle_{E \oplus F}= \langle (f_1,g) | i(f_2) \rangle_{E \oplus F}.$$
        Because $\Gamma$ is a $\dagger$-functor by Lemma~\ref{lem:swandaggerfunctor}, we conclude that $\Gamma(\rho)^\dagger=\Gamma(i).$
    \end{proof}

\begin{lem}\label{lem:muisunitary}
    If $E$ and $F$ are any vector bundles equipped with Hermitian metrics, then $\mu_{E,F}$ is a unitary isomorphism. 
\end{lem}

\begin{proof}
    Using Lemma~\ref{lem:trivialsummand}, it is clear that $E \otimes F$ is a direct summand of the trivial bundle $(E \oplus E^\perp) \otimes (F \oplus F^\perp)$. Consider the diagram
        \begin{equation}\label{diagram1}
            \begin{tikzcd}
                \Gamma((E \oplus E^\perp) \otimes (F \oplus F^\perp)) & \arrow[l, "\mu"']\Gamma(E \oplus E^\perp) \otimes_{C(T)} \Gamma(F \oplus F^\perp) \\
                \Gamma(E \otimes F) \arrow[u, "\Gamma(i)"] & \arrow[l, "\mu_{E,F}"] \Gamma(E) \otimes_{C(T)} \Gamma(F), \arrow[u, "\Gamma(i) \otimes \Gamma(i)"']
            \end{tikzcd}
        \end{equation}
        where $\mu$ is the unitary isomorphism for $(E \oplus E^\perp)$ and $(F \oplus F^\perp)$ in Proposition~\ref{lem:trivialunitaryiso}. Note that, by Lemma~\ref{lem:daggerinclusion}, $\Gamma(i)$ is injective. Furthermore,
        $$(\Gamma(\rho) \otimes_{C(T)} \Gamma(\rho))\circ (\Gamma(i) \otimes_{C(T)} \Gamma(i))  = \Gamma(\id) \otimes_{C(T)} \Gamma(\id)$$
        and so $\Gamma(i) \otimes \Gamma(i)$ is injective. Thus $\mu_{E,F}$ must be injective. Similarly, we have 
        \begin{equation}\label{diagram2}
            \begin{tikzcd}
                \Gamma((E \oplus E^\perp) \otimes (F \oplus F^\perp)) \arrow[d, "\Gamma(\rho)"'] & \arrow[l, "\mu"']\Gamma(E \oplus E^\perp) \otimes_{C(T)} \Gamma(F \oplus F^\perp) \arrow[d, "\Gamma(\rho) \otimes \Gamma(\rho)"]\\
                \Gamma(E \otimes F)  & \arrow[l, "\mu_{E,F}"] \Gamma(E) \otimes_{C(T)} \Gamma(F) ,
            \end{tikzcd}
        \end{equation}
        and so $\Gamma(\rho)$ and $\Gamma(\rho) \otimes \Gamma(\rho)$ are both surjective. This guarantees that $\mu_{E,F}$ is surjective as well. What remains to show is that $\mu_{E,F}$ is adjointable (and unitary), but this follows from our earlier work. Observe that Diagram~\eqref{diagram1} says that
        $$\Gamma(i) \circ \mu_{E,F}= \mu \circ (\Gamma(i)\otimes_{C(T)} \Gamma(i)),$$
        which implies
        $$\mu_{E,F}= \Gamma(\rho)\circ \Gamma(i) \circ \mu_{E,F}=\Gamma(\rho)\circ \mu \circ (\Gamma(i)\otimes_{C(T)}\Gamma(i))$$
        and so $\mu_{E,F}$ is a composition of adjointable maps and is therefore itself adjointable (noting that the Hermitian metrics on $E \oplus E^\perp$, $F \oplus F^\perp$, and $(E \oplus E^\perp)\otimes (F \oplus F^\perp)$ may be chosen to be compatible with Proposition~\ref{lem:trivialunitaryiso} so that $\mu$ is unitary). To verify that $\mu_{E,F}$ is unitary, we have
        \begin{align*}
            \mu_{E,F}^\dagger &= (\Gamma(i)\otimes_{C(T)} \Gamma(i))^\dagger \circ \mu^\dagger \circ \Gamma(\rho)^\dagger\\
            &= (\Gamma(i)^\dagger \otimes_{C(T)} \Gamma(i)^\dagger) \circ \mu^\dagger \circ \Gamma(\rho)^\dagger \\
            &= (\Gamma(\rho) \otimes_{C(T)}\Gamma(\rho)) \circ \mu^{-1} \circ \Gamma(i).
        \end{align*}
        But, Diagram~\eqref{diagram2} says that
        $$\mu_{E,F}^{-1}\circ \Gamma(\rho) = (\Gamma(\rho) \otimes_{C(T)} \Gamma(\rho)) \circ \mu^{-1},$$
        and so precomposing with $\Gamma(i)$ yields
        $$\mu_{E,F}^{-1} = \mu_{E,F}^{-1} \circ \Gamma(\rho) \circ \Gamma(i) = (\Gamma(\rho) \otimes_{C(T)} \Gamma(\rho)) \circ \mu^{-1} \circ \Gamma(i)=\mu_{E,F}^\dagger$$
        which proves that $\mu_{E,F}$ is unitary.
\end{proof}
    
\begin{thm}
The space of sections construction $E \mapsto \Gamma(E)$ assembles into a monoidal equivalence of C*-$\otimes$-categories 
$$\Gamma \colon \mathsf{Hilb_{fd}}(T) \to \mathsf{Hilb_{fgp}C(T)}.$$
\end{thm}

\begin{proof}
We saw in Lemma~\ref{lem:swandaggerfunctor} that $\Gamma$ is a $\dagger$-functor. We also know that $\eta$ and $\mu_{E,F}$ are unitary isomorphisms from Lemma~\ref{lem:muisunitary}. We only need to verify the coherences for a monoidal functor. It is clear from the construction of $\mu_{E,F}$ and the fact that associators $\alpha$ are determined by reparenthesizing simple tensors that the associativity diagram \eqref{eq:assoc} commutes. Indeed, for $e \in \Gamma(E)$, $f \in \Gamma(F)$, and $g \in \Gamma(G)$, observe
     \begin{center}
        \begin{tikzcd}
            (e \otimes_{C(T)} f) \otimes_{C(T)} g \arrow[r,mapsto, "\alpha"] \arrow[d, mapsto,"\mu_{E,F} \otimes \id"'] & e \otimes_{C(T)} (f \otimes_{C(T)} g) \arrow[d, mapsto,"\id \otimes \mu_{F,G}"]\\
            (e \otimes f) \otimes_{C(T)} g \arrow[d, mapsto,"\mu_{E \otimes F, G}"'] & e \otimes_{C(T)}(f \otimes g) \arrow[d, mapsto,"\mu_{E, F \otimes G}"] \\
            (e \otimes f) \otimes g \arrow[r, mapsto,"\Gamma(\alpha)"'] & e \otimes (f \otimes g)
        \end{tikzcd}
    \end{center}
    
We next need to check Diagrams \eqref{eq:leftunital} and \eqref{eq:rightunital} for unitality. For $h \in C(T)$ and $e \in \Gamma(E)$, observe


    \begin{center}

        \begin{tikzcd}
            h(t) \otimes_{C(T)} e(t) \arrow[r, mapsto,"\eta \otimes \id"] \arrow[d, mapsto,"\lambda_{\Gamma(E)}"'] & (t,h(t)) \otimes_{C(T)} e(t) \arrow[d, mapsto,"\mu_{(T \times \bC), E}"] \\
            h(t)e(t) & (t,h(t)) \otimes e(t) \arrow[l, mapsto, "\Gamma(\lambda_{E})"]
        \end{tikzcd}
    \end{center}

        \begin{center}
        \begin{tikzcd}
            e(t) \otimes_{C(T)} h(t) \arrow[r, mapsto,"\id \otimes \eta "] \arrow[d, mapsto,"\rho_{\Gamma(E)}"'] &   e(t) \otimes_{C(T)} (t,h(t)) \arrow[d, mapsto,"\mu_{E, (T \times \bC)}"] \\
            e(t)h(t) & e(t) \otimes (t,h(t)) \arrow[l, mapsto,"\Gamma(\rho_{E})"]
        \end{tikzcd}
    \end{center}

    Thus, $\Gamma$ is a monoidal $\dagger$-equivalence $\Hilb_{\mathsf{fd}}(T) \to \Hilb_{\mathsf{fgp}}(C(T))$
\end{proof}
    
\providecommand{\bysame}{\leavevmode\hbox to3em{\hrulefill}\thinspace}
\providecommand{\MR}{\relax\ifhmode\unskip\space\fi MR }
\providecommand{\MRhref}[2]{%
  \href{http://www.ams.org/mathscinet-getitem?mr=#1}{#2}
}

\end{document}